\documentclass[11pt, twoside, leqno]{article}
\usepackage{amssymb}
\usepackage{amsmath}
\usepackage{amsthm}
\usepackage{color}
\usepackage{mathrsfs}
\usepackage{graphicx}
\usepackage{indentfirst}
\usepackage{txfonts}

\allowdisplaybreaks

\def\XXint#1#2#3{{\setbox0=\hbox{$#1{#2#3}{\int}$ }
\vcenter{\hbox{$#2#3$ }}\kern-.6\wd0}}

\def\({\left(}
\def \){ \right)}

\newtheorem{theorem}{Theorem}[section]
\newtheorem{lemma}[theorem]{Lemma}

\theoremstyle{definition}
\newtheorem{remark}[theorem]{Remark}

\renewcommand{\appendix}{\par
   \setcounter{section}{0}%
   \setcounter{subsection}{0}%
   \setcounter{subsubsection}{0}%
   \gdef\thesection{\@Alph\c@section}%
   \gdef\thesubsection{\@Alph\c@section.\@arabic\c@subsection}%
   \gdef\theHsection{\@Alph\c@section.}%
   \gdef\theHsubsection{\@Alph\c@section.\@arabic\c@subsection}%
   \csname appendixmore\endcsname
 }

\numberwithin{equation}{section}

\begin{document}

\arraycolsep=1pt

\title{\bf\Large $L^p$-improving bounds of maximal functions along planar curves
\footnotetext{\hspace{-0.35cm} 2020 {\it
Mathematics Subject Classification}. Primary 42B25;
Secondary 42B20.
\endgraf {\it Key words and phrases.} maximal function, $L^p$-improving bounds, local smoothing estimate, Fourier integral operator.
\endgraf Naijia Liu was supported by China Postdoctoral Science Foundation (No.~2022M723673). Haixia Yu was supported by Natural Science Foundation of China (No.~12201378), Guangdong Basic and Applied Basic Research Foundation (No.~2023A1515010635) and STU Scientific Research Foundation for Talents (No.~NTF21038).}}
\author{Naijia Liu and Haixia Yu\footnote{Corresponding author.}}

\date{}

\maketitle

\vspace{-0.7cm}

\begin{abstract}
 In this paper, we study the $L^p(\mathbb{R}^2)$-improving bounds, i.e., $L^p(\mathbb{R}^2)\rightarrow L^q(\mathbb{R}^2)$ estimates, of the maximal function $M_{\gamma}$ along a plane curve $(t,\gamma(t))$, where
  $$M_{\gamma}f(x_1,x_2):=\sup_{u\in [1,2]}\left|\int_{0}^{1}f(x_1-ut,x_2-u \gamma(t))\,\textrm{d}t\right|,$$
 and $\gamma$ is a general plane curve satisfying some suitable smoothness and curvature conditions. We obtain $M_{\gamma} : L^p(\mathbb{R}^2)\rightarrow L^q(\mathbb{R}^2)$ if $(\frac{1}{p},\frac{1}{q})\in \Delta\cup \{(0,0)\}$ and $(\frac{1}{p},\frac{1}{q})$ satisfying $1+(1 +\omega)(\frac{1}{q}-\frac{1}{p})>0$, where $\Delta:=\{(\frac{1}{p},\frac{1}{q}):\ \frac{1}{2p}<\frac{1}{q}\leq \frac{1}{p}, \frac{1}{q}>\frac{3}{p}-1  \}$ and $\omega:=\limsup_{t\rightarrow 0^{+}}\frac{\ln|\gamma(t)|}{\ln t}$. This result is sharp except for some borderline cases. As Hickman stated in [J. Funct. Anal. 270 (2016), pp. 560--608], this is a very different situation.
\end{abstract}

\section{Introduction}

Maximal functions are crucial objects in harmonic analysis due to their importance in harmonic analysis itself and vast applications in many other areas of mathematics such as theory of partial differential equations. The main purpose of this paper is devoted to the theory of $L^p(\mathbb{R}^2)$-improving bounds, i.e., $L^p(\mathbb{R}^2)\rightarrow L^q(\mathbb{R}^2)$ estimates, of the \emph{maximal function} $M_{\gamma}$ along the plane curve $(t,\gamma(t))$, where \begin{align}\label{eq:1.00}
M_{\gamma}f(x_1,x_2):=\sup_{u\in [1,2]}\left|\int_{0}^{1}f(x_1-ut,x_2-u \gamma(t))\,\textrm{d}t\right|.
\end{align}
As Hickman stated in \cite{H}, the $L^p(\mathbb{R}^2)$-improving bounds of $M_{\gamma}$ is a very different situation. It is also a very active research topic in harmonic analysis, and has attracted a lot of attention in the last decades. The literature devoted to the subject is so broad that it is impossible to provide complete and comprehensive bibliography. Therefore, we quote only a few papers, and refer readers to \cite{Bour86,Sch,KLO,LWJ} and the references within for more detailed discussion.

The \emph{spherical maximal function} $\mathcal{M}_{\mathbb{S}^{n-1}}$ is defined by
\begin{align}\label{eq:1.01}
\mathcal{M}_{\mathbb{S}^{n-1}}f(x):=\sup_{u\in(0,\infty)}\left|
\int_{\mathbb{S}^{n-1}}f(x-uy)
\,\textrm{d}\sigma(y)\right|,
\end{align}
where $\textrm{d}\sigma$ is the surface measure on $\mathbb{S}^{n-1}$. In 1976, Stein \cite{Ste} obtained the $L^p(\mathbb{R}^n)$ boundedness of $\mathcal{M}_{\mathbb{S}^{n-1}}$ if  $p>\frac{n}{n-1}$ with $n\geq 3$. He also obtained that no such boundedness can hold for $p\leq \frac{n}{n-1}$ with $n\geq 2$. Since $\mathcal{M}_{\mathbb{S}^{1}}$ is not bounded on $L^2(\mathbb{R}^2)$ by simple examples, then the $n=2$ case is more complicated. Later, when $n=2$, Bourgain \cite{Bour86} settled this problem, and he proved the $L^p(\mathbb{R}^2)$ boundedness of $\mathcal{M}_{\mathbb{S}^{1}}$ for all $p>2$. Mockenhaupt, Seeger and Sogge \cite{MSS,MSS92} found a new proof of this boundedness by using their local smoothing estimates.

It is natural to study the $L^p(\mathbb{R}^n)\rightarrow L^q(\mathbb{R}^n)$ boundedness of $\mathcal{M}_{\mathbb{S}^{n-1}}$. However, there is no such boundedness for $\mathcal{M}_{\mathbb{S}^{n-1}}$ unless $p=q$. Now, we modify the definition in \eqref{eq:1.01} and define
\begin{align}\label{eq:1.02}
\mathcal{\bar{M}}_{\mathbb{S}^{n-1}}f(x):=\sup_{u\in [1,2]}\left|
\int_{\mathbb{S}^{n-1}}f(x-uy)
\,\textrm{d}\sigma(y)\right|.
\end{align}
Thanks to the supremum in \eqref{eq:1.02} taken over $[1,2]$, $\mathcal{\bar{M}}_{\mathbb{S}^{n-1}}$ is actually bounded from $L^p(\mathbb{R}^n)$ to $L^q(\mathbb{R}^n)$ for some $q>p$. This phenomenon is called $L^p(\mathbb{R}^n)$-improving. Schlag \cite{Sch} established the $L^p(\mathbb{R}^2)\rightarrow L^q(\mathbb{R}^2)$ boundedness of $\mathcal{\bar{M}}_{\mathbb{S}^{1}}$ for any $(\frac{1}{p},\frac{1}{q})$ lies in the interior of the triangle with vertices $(\frac{2}{5},\frac{1}{5})$, $(\frac{1}{2},\frac{1}{2})$ and $(0,0)$. Of course, it is also bounded when $(\frac{1}{p},\frac{1}{q})$ lies on the half open line connecting $(\frac{1}{2},\frac{1}{2})$ and $(0,0)$. Therefore, $\mathcal{\bar{M}}_{\mathbb{S}^{1}}$ is bounded on $\Delta\cup (0,0)$, where the definition of $\Delta$ can be found in the following Theorem \ref{thm1}. Schlag also proved that this boundedness is sharp except for some endpoints. The endpoint estimates of $\mathcal{\bar{M}}_{\mathbb{S}^{1}}$ can be found in Lee \cite{Lee03}. He showed $\mathcal{\bar{M}}_{\mathbb{S}^{1}} : L^p(\mathbb{R}^2)\rightarrow L^q(\mathbb{R}^2)$ if $(\frac{1}{p},\frac{1}{q})$ lies on the open line connecting $(\frac{2}{5},\frac{1}{5})$ and $(\frac{1}{2},\frac{1}{2})$, or connecting $(\frac{2}{5},\frac{1}{5})$ and $(0,0)$. Schlag and  Sogge \cite{SchS} characterized the $L^p(\mathbb{R}^n)\rightarrow L^q(\mathbb{R}^n)$ boundedness of $\mathcal{\bar{M}}_{\mathbb{S}^{n-1}}$ for $n\geq 3$ up to the borderline cases.

The study of $\mathcal{\bar{M}}_{\mathbb{S}^{n-1}}$ was later extended to cover more general and diverse situations: variable coefficient settings (see, for example, \cite{Sog,SchS,Iose}); Heisenberg radial function settings (see, for example, \cite{BGuoHS,LLee}); hypersurface settings (see, for example, \cite{SSte,IKM,LYan}); taking supremum over a set $E$ (see, for example, \cite{RS,AHRS}) and so on.

What's more, the \emph{maximal function} $\mathcal{M}$ defined by averages over curve $(t,\gamma(t))$,
\begin{align}\label{eq:1.03}
\mathcal{M}f(x_1,x_2):=\sup_{\epsilon\in(0,\infty)}
\frac{1}{2\epsilon}\int_{-\epsilon}^{\epsilon}\left|f(x_1-t,x_2-\gamma(t))\right|
\,\textrm{d}t,
\end{align}
is also been extensively studied, which is also a classical area of harmonic analysis. For $\gamma(t):=t^2$, Nagel, Riviere and Wainger \cite{NRW76} obtained the $L^p(\mathbb{R}^2)$ boundedness of $\mathcal{M}$ for any $p>1$. Stein \cite{Ste1} showed this boundedness for homogeneous curves and Stein and Wainger \cite{SW1} for smooth curves. Later it was extended to more general families of curves; see, for example, \cite{SW,CCVWW,CVWW}. There are some results which extend the aforementioned results to variable coefficient settings; see, for example, \cite{CWW,BJ,MR,GHLR,LYu}. In particular, in \cite{LSY}, the authors of this paper and Song proved that the \emph{maximal function}
\begin{align*}
\sup_{u\in (0,\infty)}\mathcal{M}^\infty_{u,\gamma}f(x_1,x_2):=\sup_{u\in (0,\infty)}\sup_{\epsilon \in (0,\infty)}
\frac{1}{2\varepsilon}\int_{-\varepsilon}^{\varepsilon}\left|f(x_1-t,x_2-u\gamma(t))\right|
\,\textrm{d}t
\end{align*}
is bounded on $L^p(\mathbb{R}^2)$ if and only if $p\in(2,\infty]$ for $\gamma$ satisfying some conditions. We observe that by a dilation argument, the maximal function $\mathcal{M}$ is not bounded from $L^p(\mathbb{R}^2)$ to $ L^q(\mathbb{R}^2)$ if $p\neq q$, which further implies that the maximal functions $\sup_{u\in (0,\infty)}\mathcal{M}^\infty_{u,\gamma}$ and $\sup_{u\in [1,2]}\mathcal{M}^\infty_{u,\gamma}$ are not bounded from $L^p(\mathbb{R}^2)$ to $ L^q(\mathbb{R}^2)$ if $p\neq q$.

Based on this work \cite{LSY}, combining the $L^p(\mathbb{R}^n)$-improving estimates for $\mathcal{\bar{M}}_{\mathbb{S}^{n-1}}$, it is natural to study the $L^p(\mathbb{R}^2)$-improving estimates for $M_{\gamma}$ defined in \eqref{eq:1.00}. Indeed, for some finite type curves, Li, Wang and Zhai \cite{LWZ} considered some similar maximal functions and established corresponding $L^p(\mathbb{R}^2)$-improving estimates. However, our results will include some other curves. For example, $\gamma(t):=t^{\frac{1}{2}}+t$, or $\gamma(t):=t^{\frac{1}{2}}\ln(1+t)$. We now state our main results.

\begin{theorem}\label{thm1}
Assume $\gamma\in C^{N}(0,1]$ with $N\in\mathbb{N}$ large enough, $\lim_{t\rightarrow 0^+}\gamma(t)=0$ and $\gamma$ is monotonic on $(0,1]$. Moreover, $\gamma$ satisfies the following two conditions:
\begin{enumerate}\label{curve gamma}
  \item[\rm(i)] there exist positive constants $\{C^{(j)}_{1}\}_{j=1}^{2}$ such that $|\frac{t^{j}\gamma^{(j)}(t)}{\gamma(t)}|\geq C^{(j)}_{1}$  for any $t\in (0,1]$;
  \item[\rm(ii)] there exist positive constants $\{C^{(j)}_{2}\}_{j=1}^{N}$ such that $|\frac{t^{j}\gamma^{(j)}(t)}{\gamma(t)}|\leq C^{(j)}_{2}$ for any $t\in (0,1]$.
\end{enumerate}
Then, there exists a positive constant $C$ such that for all $f\in L^{p}(\mathbb{R}^{2})$,
\begin{align*}
\left\|M_{\gamma}f\right\|_{L^{q}(\mathbb{R}^{2})} \leq C \|f\|_{L^{p}(\mathbb{R}^{2})},
\end{align*}
if $(\frac{1}{p},\frac{1}{q})\in \Delta\cup \{(0,0)\}$ and $(\frac{1}{p},\frac{1}{q})$ satisfying $1+(1 +\omega)(\frac{1}{q}-\frac{1}{p})>0$. Here and hereafter, $\Delta:=\{(\frac{1}{p},\frac{1}{q}):\ \frac{1}{2p}<\frac{1}{q}\leq \frac{1}{p}, \frac{1}{q}>\frac{3}{p}-1  \}$ and $\omega:=\limsup_{t\rightarrow 0^{+}}\frac{\ln|\gamma(t)|}{\ln t}$.
 \end{theorem}

We show the necessity of the regions of $(\frac{1}{p}, \frac{1}{q})$ in Theorem \ref{thm1}, which means our result Theorem \ref{thm1} is sharp except for some borderline cases.

\begin{theorem}\label{thm2}
Let $\gamma$ be defined as above. Then, there exists a positive constant $C$ such that for all $f\in L^{p}(\mathbb{R}^{2})$, the estimate
\begin{align*}
\left\|M_{\gamma}f\right\|_{L^{q}(\mathbb{R}^{2})} \leq C \|f\|_{L^{p}(\mathbb{R}^{2})}
\end{align*}
holds, only if the following conditions are satisfied:
\begin{enumerate}\label{necessary}
  \item[\rm(i)] $(\frac{1}{p},\frac{1}{q})$ satisfy $\frac{1}{2p}\leq\frac{1}{q}$;
  \item[\rm(ii)] $(\frac{1}{p},\frac{1}{q})$ satisfy $\frac{1}{q}\leq \frac{1}{p}$;
  \item[\rm(iii)] $(\frac{1}{p},\frac{1}{q})$ satisfy $\frac{1}{q}\geq\frac{3}{p}-1$;
  \item[\rm(iv)] $(\frac{1}{p},\frac{1}{q})$ satisfy $1+(1 +\omega)(\frac{1}{q}-\frac{1}{p})\geq0$.
\end{enumerate}
 \end{theorem}

We give some remarks about these results.

\begin{remark}\label{remark 5}
Let us explain why we take supremum over $u\in[1,2]$ in $M_{\gamma}$ defined in \eqref{eq:1.00}.
\begin{enumerate}
  \item[$\bullet$] Suppose that the estimate $\|M^{*}_{\gamma}f\|_{L^{q}(\mathbb{R}^{2})} \lesssim\|f\|_{L^{p}(\mathbb{R}^{2})}$
holds for some $1\leq p,q\leq\infty$, where
\begin{align*}
M^{*}_{\gamma}f(x_1,x_2):=\sup_{u\in (0,\infty)}\left|\int_{0}^{1}f(x_1-ut,x_2-u \gamma(t))\,\textrm{d}t\right|.
\end{align*}
One must then have $p=q$. Indeed, if $f$ is replaced by $f(\lambda\cdot)$, then we obtain $\|f(\lambda\cdot)\|_{L^{p}(\mathbb{R}^{2})}=\lambda^{-\frac{2}{p}}\|f\|_{L^{p}(\mathbb{R}^{2})}$ and $\|M^{*}_{\gamma}(f(\lambda\cdot))\|_{L^{q}(\mathbb{R}^{2})}
=\|M^{*}_{\gamma}f(\lambda\cdot)\|_{L^{q}(\mathbb{R}^{2})}=\lambda^{-\frac{2}{q}}\|M^{*}_{\gamma}f\|_{L^{q}(\mathbb{R}^{2})}$, which further leads to $\lambda^{-\frac{2}{q}}\|M^{*}_{\gamma}f\|_{L^{q}(\mathbb{R}^{2})} \lesssim\lambda^{-\frac{2}{p}}\|f\|_{L^{p}(\mathbb{R}^{2})}$. Let $\lambda\rightarrow 0$ and $\lambda\rightarrow \infty$, one must have $p=q$ as desired.
  \item[$\bullet$] On the other hand, let
\begin{align*}
M^{**}_{\gamma}f(x_1,x_2):=\sup_{u\in (0,1)}\left|\int_{0}^{1}f(x_1-ut,x_2-u \gamma(t))\,\textrm{d}t\right|.
\end{align*}
Note that $M^{**}_{\gamma}$ commutes with translations, then it is not bounded from $L^p(\mathbb{R}^2)$ to $L^q(\mathbb{R}^2)$ for all $p>q$. Suppose that the estimate $\|M^{**}_{\gamma}f\|_{L^{q}(\mathbb{R}^{2})} \lesssim\|f\|_{L^{p}(\mathbb{R}^{2})}$
holds for some $1\leq p\leq q\leq\infty$, then $p=q$. Indeed, after a simple calculation, we know that $$\left\|M^{**}_{\gamma}(f(\lambda\cdot))\right\|_{L^{q}(\mathbb{R}^{2})}=\lambda^{-\frac{2}{q}}\left\| \sup_{u\in (0,\lambda)}\left|\int_{0}^{1}f(x_1-ut,x_2-u \gamma(t))\,\textrm{d}t\right| \right\|_{L^{q}(\mathbb{R}^{2})},$$
it further follows that $$\left\|M^{*}_{\gamma}f\right\|_{L^{q}(\mathbb{R}^{2})}\leq \liminf\limits_{\lambda\rightarrow\infty} \left\| \sup_{u\in (0,\lambda)}\left|\int_{0}^{1}f(x_1-ut,x_2-u \gamma(t))\,\textrm{d}t\right| \right\|_{L^{q}(\mathbb{R}^{2})}\lesssim \liminf\limits_{\lambda\rightarrow\infty}\lambda^{\frac{2}{q}-\frac{2}{p}}\|f\|_{L^{p}(\mathbb{R}^{2})}=0 $$
if $q>p$. It is a contradiction. Therefore, one must have $p=q$.
\end{enumerate}
In conclusion, it is necessary to take supremum over $u\in[1,2]$ in $M_{\gamma}$.
\end{remark}

\begin{remark}\label{remark 1}
The conditions (i) and (ii) of Theorem \ref{thm1} imply the following doubling condition of the curve (see \cite{LSY}), i.e., it holds that
 \begin{align*}
e^{C^{(1)}_{1}/2}\leq \frac{\gamma(2t)}{\gamma(t)}\leq e^{C^{(1)}_{2}}
  \end{align*}
and
\begin{align*}
\textrm{either} \quad  e^{C^{(2)}_{1}/2 C^{(1)}_{2}}\leq \frac{\gamma'(2t)}{\gamma'(t)}\leq e^{C^{(2)}_{2}/C^{(1)}_{1}}\quad  \textrm{or}\quad e^{-C^{(2)}_{2}/C^{(1)}_{1}}\leq \frac{\gamma'(2t)}{\gamma'(t)}\leq e^{-C^{(2)}_{1}/2 C^{(1)}_{2}}.
  \end{align*}
The doubling condition has already been used to obtain the $L^p(\mathbb{R}^2)$ boundedness of the maximal function $\mathcal{M}$ defined in \eqref{eq:1.03}; see, for example, \cite{CCCD,CRu}.
\end{remark}

\begin{remark}\label{remark 2}
Let us list some examples of curves satisfying the conditions (i) and (ii) of Theorem \ref{thm1}. We may add a characteristic function $\chi_{(0,\varepsilon_0]}(t)$ to these curves if necessary, where $\varepsilon_0$ is small enough.
\begin{enumerate}
\item[\rm(1)] $\gamma_1(t):=t^d$, where $d\in(0,\infty)$ and $d\neq1$;
\item[\rm(2)] $\gamma_2(t):=t^d\ln(1+t)$, where $d\in(0,\infty)$;
\item[\rm(3)] $\gamma_3(t):=a_dt^d+a_{d+1}t^{d+1}+\cdots+a_{d+m}t^{d+m}$, where $d\geq 2$, $d\in\mathbb{N}$ and $m\in\mathbb{N}_0$ , $a_d\neq 0$, i.e., $\gamma_3$ is a polynomial of degree at least $d\geq 2$ with no linear term and constant term;
\item[\rm(4)] $\gamma_4(t):=\sum_{i=1}^{d}\beta_i t^{\alpha_i}$, where $\alpha_i\in(0,\infty)$ for all $i=1,2, \cdots, d$, $\min_{i\in\{1,2, \cdots, d\}}\{\alpha_i\}_{i=1}^{d}\neq 1$ and $d\in \mathbb{N}$;
\item[\rm(5)] $\gamma_5(t):=1-\sqrt{1-t^2}$, or $t\sin t$, or $t-\sin t$, or $1-\cos t$, or $e^t-t-1$;
\item[\rm(6)] $\gamma_6(t)$ is a smooth function on $[0,1]$ satisfying $\gamma(0)=\gamma'(0)=\cdots=\gamma^{(d-1)}(0)=0$ and $\gamma^{(d)}(0)\neq0$, where $d\geq 2$ and $d\in\mathbb{N}$. Note that $\gamma_6$ is finite type $d$ at $0$ (see, \cite{Iose}), $\gamma_3$ and $\gamma_5$ are special cases of $\gamma_6$.
\end{enumerate}
\end{remark}

\begin{remark}\label{remark 3}
If $\gamma(t):=\gamma_1(t)=t^d$ with $d\in(0,\infty)$ and $d\neq1$, it follows that $1+(1 +\omega)(\frac{1}{q}-\frac{1}{p})>0$ is equivalent to
$(\frac{1}{p},\frac{1}{q})\in\{(\frac{1}{p},\frac{1}{q}):\ \frac{1}{q}>\frac{1}{p}- \frac{1}{d+1}\}$. Therefore, we can express the regions of $(\frac{1}{p}, \frac{1}{q})$ in Theorem \ref{thm1} as in Figure \ref{Figure:1}. It is easy to see that the regions of $(\frac{1}{p}, \frac{1}{q})$ in Theorem \ref{thm1} decrease as $d$ increases. In particular, when $d\in(0,4]$ and $d\neq1$, the regions of $(\frac{1}{p}, \frac{1}{q})$ in Theorem \ref{thm1} is $\Delta\cup \{(0,0)\}$. On the other hand, by a simple calculation, we can also express the regions of $(\frac{1}{p}, \frac{1}{q})$ in Theorem \ref{thm1} for $\gamma=\gamma_3$ as in Figure \ref{Figure:1}.
\begin{figure}[htbp]
  \centering
  \includegraphics[width=3.68in]{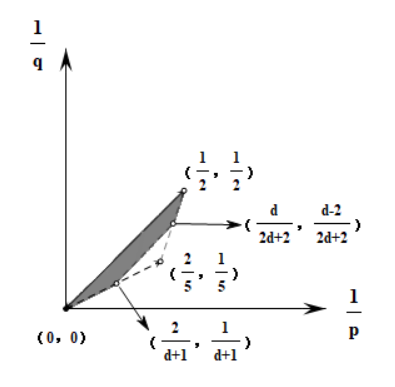}\\
  \caption{Regions of $(\frac{1}{p}, \frac{1}{q})$ in Theorem \ref{thm1} for $\gamma=\gamma_1$ or $\gamma=\gamma_3$.}\label{Figure:1}
\end{figure}
\end{remark}

We like to mention several ingredients in the proofs of Theorems \ref{thm1} and \ref{thm2}. For the general curve $\gamma$, since it does not satisfy the following special property:  $\gamma(ab)=\gamma(a)\gamma(b)$ for any $a>0$ and $b>0$, then the corresponding case will be more complicated than the homogeneous curve case. We overcome this difficulty by replacing $\gamma(2^j t)$ as $\Gamma_{j}(t):=\frac{\gamma(2^jt)}{\gamma(2^j)}$ and reduce our estimate to \eqref{eq:2.1}. From the following Lemma \ref{lemma 2.1}, we have that $\Gamma_j$ behaves uniformly in the parameter $j$. On the other hand, by using the theory of oscillatory integrals and stationary phase estimates, it is enough to obtain a local smoothing estimate for
\begin{align*}
A_{j,k}f(x,u):=\int_{\mathbb{R}^{2}} e^{i\Psi(x,u,\xi)}a_{j,k}(x,u,\xi) \hat{f}(\xi)\,\textrm{d}\xi,
\end{align*}
defined in \eqref{eq:2.22}, which is essentially a Fourier integral operator with phase function $x\cdot\xi-u|\xi|$. The local smoothing estimate for $A_{j,k}$ will be reduced to a decoupling inequality for cones due to Bourgain and Demeter \cite{BoD}. In this paper, we establish a local smoothing estimate for $A_{j,k}$ by interpolation between \eqref{eq:2.30} with \eqref{eq:2.29} and \eqref{eq:2.27}. Our proofs of \eqref{eq:2.30}, \eqref{eq:2.29} and \eqref{eq:2.27} rely on local smoothing estimates obtained in Beltran, Hickman and Sogge \cite{B} and Lee \cite{Lee}. It is worth noticing that verifying the so called cinematic curvature condition in \cite{B,Lee} is another highlight of this paper, which need some complicated calculations. In order to obtain the necessity of the regions of $(\frac{1}{p}, \frac{1}{q})$, we construct some examples and introduce a slightly different
notation $\omega=\limsup_{t\rightarrow 0^{+}}\frac{\ln|\gamma(t)|}{\ln t}$ to the general plane curve $(t,\gamma(t))$, which shows that the regions of $(\frac{1}{p}, \frac{1}{q})$ is related to $\gamma$.

The layout of the paper is as follows. In Section 2, we show Theorem \ref{thm1}, whose proof relies  heavily on the local smoothing estimates. In Section 3, we consider the necessary conditions for Theorem \ref{thm1}, i.e., Theorem \ref{thm2}. We will see that the $L^p(\mathbb{R}^2)\rightarrow L^q(\mathbb{R}^2)$ boundedness in Theorem \ref{thm1} is almost sharp except for some endpoints.

Finally, we make some convention on notation. Throughout this paper, the letter ``$C$"  will denote a \emph{positive constant}, independent of the essential variables, but whose value may change at each occurrence. $a\lesssim b$ (or $a\gtrsim b$) means that there exists a positive constant $C$ such that $a\leq Cb$ (or $a\geq Cb$). $a\approx b$ means $a\lesssim b$ and $b\lesssim a$. For any $x\in\mathbb{R}^n$ and $r\in(0,\infty)$, let $B(x,r):=\{y\in\mathbb{R}^n:\ |x-y|<r\}$ and $B^{\complement}(x,r)$
be its \emph{complement} in $\mathbb{R}^n$. $\hat{f}$ and $f^{\vee}$ shall denote the \emph{Fourier transform} and the \emph{inverse Fourier transform} of $f$, respectively. For $1<q\leq\infty$, we will denote $q'$ the \emph{adjoint number} of $q$, i.e., ${1}/{q}+ {1}/{q'}=1$. Let $\mathbb{N}:=\{1,\,2,...\}$ and $\mathbb{N}_0:=\mathbb{N}\bigcup\{0\}$, $\mathbb{R}^{+}:=(0,\infty)$. For any set $E$, we use $\chi_E$ to denote the \emph{characteristic function} of $E$.

\section{Proof of Theorem \ref{thm1}}

In this section, we devote to the proof of Theorem \ref{thm1}. We first introduce some lemmas which will be used in the proof of Theorem \ref{thm1}.

\begin{lemma}\label{lemma 2.1}
(\cite[Lemma 2.2]{LSY}) Let $t\in [\frac{1}{2},2]$ and $\Gamma_{j}(t):=\frac{\gamma(2^jt)}{\gamma(2^j)}$ with $j\in \mathbb{Z}$. We have the following inequalities hold uniformly in $j$,
\begin{enumerate}
  \item[\rm(i)] $  e^{-C^{(1)}_{2}}\leq\Gamma_{j}(t)\leq e^{C^{(1)}_{2}}$;
  \item[\rm(ii)] $ \frac{C^{(1)}_1}{2e^{C^{(1)}_{2}}}\leq|\Gamma_{j}'(t)|\leq 2e^{C^{(1)}_{2}}C^{(1)}_2$;
  \item[\rm(iii)] $\frac{C^{(2)}_1}{4e^{C^{(1)}_{2}}}\leq|\Gamma_{j}''(t)|\leq 4e^{C^{(1)}_{2}}C^{(2)}_{2}$;
  \item[\rm(iv)] $|\Gamma_{j}^{(k)}(t)|\leq  2^{k} e^{C^{(1)}_{2}}C^{(k)}_{2}$ \quad \textrm{for} \textrm{all} $2\leq k\leq N$ \textrm{and} $k\in \mathbb{N}$;
  \item[\rm(v)] $|((\Gamma_{j}')^{-1})^{(k)}(t)|\lesssim 1$ \quad \textrm{for} \textrm{all} $0\leq k< N$  \textrm{and} $k\in \mathbb{N}$, \textrm{where} $(\Gamma_{j}')^{-1}$ \textrm{is} \textrm{the} \textrm{inverse} \textrm{function} \textrm{of} $\Gamma_{j}'$.
\end{enumerate}
\end{lemma}

The following lemma is well known.

\begin{lemma}\label{lemma 2.2}
(see, for example, \cite[Lemma 2.4.2]{So} or \cite[Lemma 43]{BHS}) Suppose that $F$ is $C^1(\mathbb{R})$. Then, if $q>1$ and $\frac{1}{q}+\frac{1}{q'}=1$,
\begin{align*}
\sup_{u\in [1,2]}|F(u)|^q\leq |F(1)|^q+q\left(\int_1^2|F(u)|^q\,\textrm{d}u \right)^{\frac{1}{q'}}\left(\int_1^2|F'(u)|^q\,\textrm{d}u \right)^{\frac{1}{q}}.
\end{align*}
\end{lemma}

\begin{lemma}\label{lemma 2.10}
Recall that $\omega=\limsup_{t\rightarrow 0^{+}}\frac{\ln|\gamma(t)|}{\ln t}$. Then, for all $(\frac{1}{p},\frac{1}{q})$ satisfying $\frac{1}{q}\leq \frac{1}{p}$ and $1+(1 +\omega)(\frac{1}{q}-\frac{1}{p})>0$, we have
\begin{align*}
\sum_{j\leq 0} 2^j|2^j\gamma(2^j)|^{\frac{1}{q}-\frac{1}{p}}<\infty.
\end{align*}
\end{lemma}
\begin{proof}[Proof (of Lemma \ref{lemma 2.10}).] We first show $\omega\in (0,\infty)$. Indeed, notice the facts that $\lim_{t\rightarrow 0^+}\gamma(t)=0$ and $C^{(1)}_{1}\leq|\frac{t\gamma'(t)}{\gamma(t)}|\leq C^{(1)}_{2}$ for any $t\in (0,1]$, by l'H\^{o}pital's Rule, it is easy to see that $\omega\in (0,\infty)$. From $1+(1 +\omega)(\frac{1}{q}-\frac{1}{p})>0$, then there exists $\epsilon\in (0,\infty)$ such that $1+(1 +\omega+\epsilon)(\frac{1}{q}-\frac{1}{p})>0$. By the definition of $\omega$, we have $ \inf_{\tau\in (0,1)}\sup_{t\in (0,\tau)}\frac{\ln|\gamma(t)|}{\ln t}=\omega$. Then, for this $\epsilon$, there exists a positive constant $\tau\in (0,1)$ such that $\sup_{t\in (0,\tau)}\frac{\ln|\gamma(t)|}{\ln t}\in [\omega,\omega+\epsilon)$, which implies that $\frac{\ln|\gamma(t)|}{\ln t}<\omega+\epsilon$ for all $t\in (0,\tau)$. Therefore,
\begin{align*}
|\gamma(t)|=t^{\log_t {|\gamma(t)|}}=t^{\frac{\ln |\gamma(t)|}{\ln t}}>t^{\omega+\epsilon}
\end{align*}
for all $t\in (0,\tau)$. By $\frac{1}{q}\leq \frac{1}{p}$, we deduce that
\begin{align*}
\sum_{j\leq 0, ~2^j<\tau} 2^j|2^j\gamma(2^j)|^{\frac{1}{q}-\frac{1}{p}}< \sum_{j\leq 0, ~2^j<\tau} 2^j\left(2^j(2^j)^{\omega+\epsilon}\right)^{\frac{1}{q}-\frac{1}{p}}=\sum_{j\leq 0, ~2^j<\tau} 2^{j\left[1+(1 +\omega+\epsilon)\left(\frac{1}{q}-\frac{1}{p}\right)\right]} <\infty,
\end{align*}
which further leads to $\sum_{j\leq 0} 2^j|2^j\gamma(2^j)|^{\frac{1}{q}-\frac{1}{p}}<\infty$ as required.
\end{proof}

We now turn to the proof of Theorem \ref{thm1}. Since the maximal function
$$M_{\gamma}f(x_1,x_2)=\sup_{u\in [1,2]}\left|\int_{0}^{1}f(x_1-ut,x_2-u \gamma(t))\,\textrm{d}t\right|$$
that we are dealing can be seemed as a positive operator, we may assume that $f\geq 0$. For any $u\in [1,2]$, denote
$$M_{u,\gamma}f(x_1,x_2):=\int_{0}^{1}f(x_1-ut,x_2-u \gamma(t))\,\textrm{d}t,$$
the first step is to break up $M_{u,\gamma}$ into pieces by a standard partition of unity. Let $\psi:\ \mathbb{R}^+\rightarrow\mathbb{R}$ be a smooth function supported on $\{t\in \mathbb{R}:\ \frac{1}{2}\leq t\leq 2\}$ with the property that $0\leq \psi(t)\leq 1$ and $\Sigma_{j\in \mathbb{Z}} \psi_j(t)=1$ for any $t> 0$, where $\psi_j(t):=\psi (2^{-j}t)$. We have
$$M_{u,\gamma}f(x_1,x_2)\leq\sum_{j\leq 0}\int_{0}^{\infty}f(x_1-ut,x_2-u \gamma(t))\psi_j(t)\,\textrm{d}t=: \sum_{j\leq 0} M_{u,\gamma,j}f(x_1,x_2).$$
Via a change of variables, we rewrite $M_{u,\gamma,j}f(x_1,x_2)$ as
$$\int_{0}^{\infty}f(x_1-u2^jt,x_2-u \gamma(2^jt))\psi(t)2^j\,\textrm{d}t.$$

Let
$$\widetilde{M}_{u,\Gamma_j}f(x_1,x_2):=\int_{0}^{\infty}f(x_1-ut,x_2-u \Gamma_j(t))\psi(t)\,\textrm{d}t$$
with $\Gamma_{j}(t)=\frac{\gamma(2^jt)}{\gamma(2^j)}$, and for any $j\in \mathbb{Z}$ define
$$\delta_jf(x_1,x_2):=f(2^jx_1,\gamma(2^j)x_2).$$
Note that
$$|2^j\gamma(2^j)|^{\frac{1}{q}}\|\delta_jf\|_{L^{q}(\mathbb{R}^{2})}=\|f\|_{L^{q}(\mathbb{R}^{2})}$$
and
$$\delta_j\left(\sup_{u\in [1,2]} M_{u,\gamma,j}f\right)=2^j \left(\sup_{u\in [1,2]} \widetilde{M}_{u,\Gamma_j}\right)(\delta_jf),$$
then it is enough to prove

\begin{align}\label{eq:2.1}
\sum_{j\leq 0} 2^j|2^j\gamma(2^j)|^{\frac{1}{q}-\frac{1}{p}} \left\|\sup_{u\in [1,2]} \widetilde{M}_{u,\Gamma_j} \right\|_{L^{p}(\mathbb{R}^{2})\rightarrow L^{q}(\mathbb{R}^{2})}\lesssim 1.
\end{align}

After taking a Fourier transform, we see that
$$\widetilde{M}_{u,\Gamma_j}f(x)=\int_{\mathbb{R}^{2}} \hat{f}(\xi) e^{ix\cdot\xi}H(u, \xi)\,\textrm{d}\xi $$
with $x:=(x_1, x_2)$ and $\xi:=(\xi_1, \xi_2)$, where
$$H(u, \xi):=\int_{\mathbb{R}}  e^{-iu\xi_1 t-i u \xi_2 \Gamma_j(t) }\psi(t)\,\textrm{d}t.$$

Recall that  $\psi:\ \mathbb{R}^+\rightarrow\mathbb{R}$ is a smooth function supported on $\{t\in \mathbb{R}:\ \frac{1}{2}\leq t\leq 2\}$ with $\Sigma_{j\in \mathbb{Z}} \psi_j(t)=1$ for any $t> 0$, let
$$ \widetilde{M}_{u,\Gamma_j,k}f(x):=\int_{\mathbb{R}^{2}} \hat{f}(\xi) e^{ix\cdot\xi}H(u, \xi)\psi_k(u|\xi|)\,\textrm{d}\xi,$$
and
$$\widetilde{M}^0_{u,\Gamma_j}f(x):= \int_{\mathbb{R}^{2}} \hat{f}(\xi) e^{ix\cdot\xi}H(u, \xi)\Psi^0(u|\xi|)\,\textrm{d}\xi$$
with $\Psi^0(t):= \sum_{k\leq 0}\psi_k(t)$, we further make the following decomposition:
\begin{align}\label{eq:2.2}
\widetilde{M}_{u,\Gamma_j}f(x)= \widetilde{M}^0_{u,\Gamma_j}f(x)+ \sum_{k\geq 1} \widetilde{M}_{u,\Gamma_j,k}f(x).
\end{align}

Consider $\sup_{u\in [1,2]} |\widetilde{M}^0_{u,\Gamma_j}f|$. Noting that $\Psi^0$ is a smooth function supported on $\{t\in \mathbb{R}:\ 0\leq t\leq 2\}$, it is easy to see that
$$\left|H(u, \xi)\Psi^0(u|\xi|)\right|\lesssim \frac{1}{(1+|\xi|)^3}.$$
Furthermore, together with Lemma \ref{lemma 2.1} and the fact that $\Psi^0$ is a compactly supported smooth function, enable us to obtain
\begin{align}\label{eq:2.3}
\left|\partial_u^{\alpha}\partial^{\beta}_{\xi}\left(H(u, \xi)\Psi^0(u|\xi|)\right)\right|\lesssim \frac{1}{(1+|\xi|)^3}
\end{align}
for all $(\alpha,\beta)\in \mathbb{N}_0\times \mathbb{N}^2_0$ with $|\alpha|\leq 1$ and $|\beta|\leq 3$, where the implicit constant is independent of $u\in[1,2]$. From \eqref{eq:2.3}, it is not difficult to see that
\begin{eqnarray}\label{eq:2.4}
\left\{\aligned
\left(H(u, \cdot)\Psi^0(u|\cdot|)\right)^{\vee}&\in L^{r}(\mathbb{R}^{2});\\
\left[\partial _u\left( H(u, \cdot)\Psi^0(u|\cdot|)\right)\right]^{\vee}& \in L^{r}(\mathbb{R}^{2}).
\endaligned\right.
\end{eqnarray}
for all $1\leq r \leq \infty$.

We now turn to bound the $L^q(\mathbb{R}^2)$ norm of $\sup_{u\in [1,2]} |\widetilde{M}^0_{u,\Gamma_j}f|$. By Lemma \ref{lemma 2.2} and $\textrm{H}\ddot{\textrm{o}}\textrm{lder}$'s inequality, it is easy to verify that
\begin{align}\label{eq:2.99}
&\left\|\sup_{u\in [1,2]}\left|\widetilde{M}^0_{u,\Gamma_j}f\right|\right\|^q_{L^{q}(\mathbb{R}^{2})} =\int_{\mathbb{R}^2} \sup_{u\in [1,2]}\left|\widetilde{M}^0_{u,\Gamma_j}f\right|^q  \,\textrm{d}x\\
&\quad \quad\leq \int_{\mathbb{R}^2} \left|\widetilde{M}^0_{1,\Gamma_j}f\right|^q  \,\textrm{d}x+q \int_{\mathbb{R}^2} \left(\int_1^2 \left|\widetilde{M}^0_{u,\Gamma_j}f\right|^q\,\textrm{d}u\right)^{\frac{1}{q'}}\left(\int_1^2 \left|\partial _u \left(\widetilde{M}^0_{u,\Gamma_j}f\right)\right|^q\,\textrm{d}u\right)^{\frac{1}{q}} \,\textrm{d}x  \nonumber \\
& \quad\quad\leq\int_{\mathbb{R}^2} \left|\widetilde{M}^0_{1,\Gamma_j}f\right|^q  \,\textrm{d}x+q \left(\int_1^2\int_{\mathbb{R}^2}  \left|\widetilde{M}^0_{u,\Gamma_j}f\right|^q \,\textrm{d}x\,\textrm{d}u\right)^{\frac{1}{q'}}\left(\int_1^2\int_{\mathbb{R}^2}  \left|\partial _u \left(\widetilde{M}^0_{u,\Gamma_j}f\right)\right|^q \,\textrm{d}x\,\textrm{d}u\right)^{\frac{1}{q}}.\nonumber
\end{align}
Moreover, by Young's inequality, the estimate \eqref{eq:2.99} together with \eqref{eq:2.4} gives
\begin{align}\label{eq:2.6}
\left\|\sup_{u\in [1,2]}\left|\widetilde{M}^0_{u,\Gamma_j}f\right|\right\|_{L^{q}(\mathbb{R}^{2})} \lesssim \|f\|_{L^{p}(\mathbb{R}^{2})}
\end{align}
for all $q\geq p \geq 1$.

Notice that $(\frac{1}{p},\frac{1}{q})$ satisfy $\frac{1}{q}\leq \frac{1}{p}$ and $1+(1 +\omega)(\frac{1}{q}-\frac{1}{p})>0$, by Lemma \ref{lemma 2.10}, we have $\sum_{j\leq 0} 2^j|2^j\gamma(2^j)|^{\frac{1}{q}-\frac{1}{p}} \lesssim 1$. This, combined with \eqref{eq:2.6}, which trivially leads to the estimate
\begin{align}\label{eq:2.7}
\sum_{j\leq 0} 2^j|2^j\gamma(2^j)|^{\frac{1}{q}-\frac{1}{p}} \left\|\sup_{u\in [1,2]} \left|\widetilde{M}^0_{u,\Gamma_j}\right| \right\|_{L^{p}(\mathbb{R}^{2})\rightarrow L^{q}(\mathbb{R}^{2})}\lesssim 1.
\end{align}
This is the desired estimate for the first part. We turn to the second part.

Consider $\sup_{u\in [1,2]} |\widetilde{M}_{u,\Gamma_j,k}f|$. Recall that
$$ \widetilde{M}_{u,\Gamma_j,k}f(x)=\int_{\mathbb{R}^{2}} \hat{f}(\xi) e^{ix\cdot\xi}H(u, \xi)\psi_k(u|\xi|)\,\textrm{d}\xi$$
and
$$H(u, \xi)=\int_{\mathbb{R}}  e^{-iu\xi_1 t-i u \xi_2 \Gamma_j(t) }\psi(t)\,\textrm{d}t.$$
Let the phase function in $H(u, \xi)$ be
$$\varphi(u,\xi,t):=-u\xi_1 t- u \xi_2 \Gamma_j(t).$$
Differentiate in $t$ to obtain
\begin{eqnarray}\label{eq:2.01}
\left\{\aligned
&\varphi'_t(u,\xi,t)=-u\xi_1 - u \xi_2 \Gamma'_j(t);\\
&\varphi''_t(u,\xi,t)=- u \xi_2 \Gamma''_j(t);\\
&\varphi'''_t(u,\xi,t)=- u \xi_2 \Gamma'''_j(t).
\endaligned\right.
\end{eqnarray}

By $(\textrm{ii})$ of Lemma \ref{lemma 2.1}, we obtain $ \frac{C^{(1)}_1}{2e^{C^{(1)}_{2}}} \leq|\Gamma_{j}'(t)|\leq 2e^{C^{(1)}_{2}}C^{(1)}_2$. If $|\xi_1|\geq 10 e^{C^{(1)}_{2}}C^{(1)}_2 |\xi_2|$, then clearly \begin{align}\label{eq:2.8}
|\varphi'_t(u,\xi,t)|\geq \frac{|\xi_1|}{2}-2|\xi_2||\Gamma'_j(t)|\geq |\xi_1|+|\xi_2|.
\end{align}
If $|\xi_2|\geq \frac{10e^{C^{(1)}_{2}}}{C^{(1)}_{1} }|\xi_1|$, we immediately conclude that
\begin{align}\label{eq:2.9}
|\varphi'_t(u,\xi,t)|\geq |\xi_2||\Gamma'_j(t)|-2|\xi_1| \geq |\xi_1|+|\xi_2|.
\end{align}
Let $\chi \in C^{\infty}_{c}(\mathbb{R}^{+})$ be a function such that $\chi =1$ on $[\frac{C^{(1)}_{1}}{10e^{C^{(1)}_{2}}}, 10 e^{C^{(1)}_{2}}C^{(1)}_2]$, integration by parts shows that
$$\left|\left(1-\chi\left(\frac{|\xi_1|}{|\xi_2|}\right)\right)H(u, \xi) \right|\lesssim \frac{1}{(1+|\xi|)^4}.$$
Notice that $|\xi|\approx 2^k$. Consequently,
\begin{align}\label{eq:2.10}
\left|\partial^{\alpha}_{\xi} \left[\left(1-\chi\left(\frac{|\xi_1|}{|\xi_2|}\right)\right)H(u, \xi)\psi_k(u|\xi|) \right]\right|\lesssim 2^{-k}\frac{1}{(1+|\xi|)^3}
\end{align}
and
\begin{align}\label{eq:2.1010}
\left|\partial_u\partial^{\alpha}_{\xi} \left[\left(1-\chi\left(\frac{|\xi_1|}{|\xi_2|}\right)\right)H(u, \xi)\psi_k(u|\xi|) \right]\right|\lesssim \frac{1}{(1+|\xi|)^3}
\end{align}
for all $|\alpha|\leq 3$, where the implicit constant is independent of $u\in[1,2]$.

Define
\begin{eqnarray*}
\left\{\aligned
&\widetilde{M}^1_{u,\Gamma_j,k}f(x):=\int_{\mathbb{R}^{2}} \hat{f}(\xi) e^{ix\cdot\xi}\chi\left(\frac{|\xi_1|}{|\xi_2|}\right)H(u, \xi)\psi_k(u|\xi|)\,\textrm{d}\xi;\\
&\widetilde{M}^2_{u,\Gamma_j,k}f(x):=\int_{\mathbb{R}^{2}} \hat{f}(\xi) e^{ix\cdot\xi}\left(1-\chi\left(\frac{|\xi_1|}{|\xi_2|}\right)\right)H(u, \xi)\psi_k(u|\xi|)\,\textrm{d}\xi.
\endaligned\right.
\end{eqnarray*}
As in \eqref{eq:2.3} and \eqref{eq:2.4}, by \eqref{eq:2.10} and \eqref{eq:2.1010}, we may conclude that
\begin{align}\label{eq:2.10101}
\left\|\widetilde{M}^2_{u,\Gamma_j,k}f\right\|_{L^{q}(\mathbb{R}^{2})} \lesssim 2^{-k} \|f\|_{L^{p}(\mathbb{R}^{2})} ~~~\textrm{and}~~~\left\|\partial_u\left(\widetilde{M}^2_{u,\Gamma_j,k}f\right)\right\|_{L^{q}(\mathbb{R}^{2})} \lesssim  \|f\|_{L^{p}(\mathbb{R}^{2})}
\end{align}
for all $q\geq p \geq 1$. Furthermore, we obtain from \eqref{eq:2.99} that
\begin{align*}
\left\|\sup_{u\in [1,2]}\left|\widetilde{M}^2_{u,\Gamma_j,k}f\right|\right\|_{L^{q}(\mathbb{R}^{2})} \lesssim 2^{-\frac{k}{q'}} \|f\|_{L^{p}(\mathbb{R}^{2})}
\end{align*}
holds for all $q\geq p \geq 1$. This along with Lemma \ref{lemma 2.10} shows
\begin{align}\label{eq:2.11}
\sum_{j\leq 0} 2^j|2^j\gamma(2^j)|^{\frac{1}{q}-\frac{1}{p}} \sum_{k\geq 1}\left\|\sup_{u\in [1,2]} \left|\widetilde{M}^2_{u,\Gamma_j,k} \right|\right\|_{L^{p}(\mathbb{R}^{2})\rightarrow L^{q}(\mathbb{R}^{2})}\lesssim \sum_{j\leq 0} 2^j|2^j\gamma(2^j)|^{\frac{1}{q}-\frac{1}{p}} \sum_{k\geq 1} 2^{-\frac{k}{q'}} \lesssim 1.
\end{align}
This is the desired estimate for $\sup_{u\in [1,2]} |\widetilde{M}^2_{u,\Gamma_j,k}|$. Thus, it remains to consider the other maximal function $\sup_{u\in [1,2]} |\widetilde{M}^1_{u,\Gamma_j,k}|$.

For $\widetilde{M}^1_{u,\Gamma_j,k}$, we consider two situations. If $\xi_1\xi_2\geq0$, it is easy to see that
\begin{align*}
|\varphi'_t(u,\xi,t)|\gtrsim |\xi_1|+|\xi_2|.
\end{align*}
Therefore, we also have
\begin{align*}
\left|\partial^{\alpha}_{\xi}\left(\chi\left(\frac{|\xi_1|}{|\xi_2|}\right)H(u, \xi)\psi_k(u|\xi|) \right)\right|\lesssim 2^{-k}\frac{1}{(1+|\xi|)^3}
\end{align*}
for all $|\alpha|\leq 3$ with the implicit constant is independent of $u\in[1,2]$. Consequently, as for $\sup_{u\in [1,2]} |\widetilde{M}^2_{u,\Gamma_j,k}|$, we may obtain the desired estimate \eqref{eq:2.11} for $\sup_{u\in [1,2]} |\widetilde{M}^1_{u,\Gamma_j,k}|$ in this case.

From now on, we will restrict our view on the most difficult situation in which $\xi_1$ and $\xi_2$ in $\widetilde{M}^1_{u,\Gamma_j,k}$ satisfying $\xi_1\xi_2<0$. Let
$$\varphi'_t(u,\xi,t_0)=-u\xi_1 - u \xi_2 \Gamma'_j(t_0)=0,$$
then $t_0$ is the critical point and $\Gamma'_j(t_0)=-\frac{\xi_1}{\xi_2}$. We remark that $t_0$ depends on $j$ and since our estimates about $t_0$ uniformly in the parameter $j$, we omit the parameter $j$ in notation $t_0$. Notice that $\Gamma'_j$ is strictly monotonic. This is because $\gamma'$ is strictly monotonic. Consequently, we can write $t_0=(\Gamma'_j)^{-1}(-\frac{\xi_1}{\xi_2})$. Furthermore, we can assume that $t_0\in (\frac{1}{2},2)$. Otherwise, one easily sees that $|\varphi'_t(u,\xi,t)|\gtrsim |\xi_1|+|\xi_2|$ holds, which will leads to the desired estimate as in the treatment of $\sup_{u\in [1,2]} |\widetilde{M}^2_{u,\Gamma_j,k}|$.

We shall consider a one-dimensional oscillatory integral involving phase function with a non-degenerate critical point. Based on an approach in the spirit of stationary phase estimates, we first rewrite
 $$H(u, \xi)=\int_{\mathbb{R}}  e^{i \varphi(u,\xi,t) }\psi(t)\,\textrm{d}t=\int_{\mathbb{R}}  e^{i\varphi(u,\xi,t+t_0) }\psi(t+t_0)\,\textrm{d}t.$$
Applying Taylor's theorem gives
$$\varphi(u,\xi,t+t_0)=\varphi(u,\xi,t_0)+\varphi'_t(u,\xi,t_0)+\frac{t^2}{2!}\varphi''_t(u,\xi,t_0)
+\frac{t^3}{2!}\int_0^1(1-\theta)^2\varphi'''_t(u,\xi,\theta t+t_0)\,\textrm{d}\theta.$$
Let us set $$\eta(t,t_0):=\frac{\Gamma''_j(t_0)}{2}+\frac{t}{2}\int_0^1(1-\theta)^2\Gamma'''_j(\theta t+t_0)\,\textrm{d}\theta,$$ which can be seemed as a compactly supported smooth function. From \eqref{eq:2.01} and the fact that $t_0$ is the critical point, we conclude that
\begin{align}\label{eq:2.13}
\varphi(u,\xi,t+t_0)=\varphi(u,\xi,t_0)-u\xi_2t^2\eta(t,t_0).
\end{align}

Therefore, we can write
\begin{align}\label{eq:2.014}
\widetilde{M}^1_{u,\Gamma_j,k}f(x)=\int_{\mathbb{R}^{2}} e^{ix\cdot\xi}e^{i \varphi(u,\xi,t_0)}m_{j,k}(u,\xi) \hat{f}(\xi)\,\textrm{d}\xi,
\end{align}
where $$m_{j,k}(u,\xi):=\chi\left(\frac{|\xi_1|}{|\xi_2|}\right)\psi_k(u|\xi|)  \int_{\mathbb{R}}  e^{-iu\xi_2t^2\eta(t,t_0) }\psi(t+t_0)\,\textrm{d}t.$$
In order to estimate $\sup_{u\in [1,2]}| \widetilde{M}^1_{u,\Gamma_j,k}|$ in the case that $\xi_1\xi_2<0$, we first observe that $\widetilde{M}^1_{u,\Gamma_j,k}$ can be localized. Indeed, we can rewrite $\widetilde{M}^1_{u,\Gamma_j,k}$ as
\begin{align*}
\int_{\mathbb{R}^{2}}\int_{\mathbb{R}^{2}} e^{i(x-y)\cdot\xi}e^{i \varphi(u,\xi,t_0)}m_{j,k}(u,\xi) \,\textrm{d}\xi f(y) \,\textrm{d}y.
\end{align*}
Notice that $\Gamma'_j(t_0)=-\frac{\xi_1}{\xi_2}$, $t_0\approx 1$ and $\Gamma_j(t_0)\approx 1$, then there exists a positive constant $\varpi$ large enough such that
\begin{align*}
\left|\nabla_{\xi}\left[(x-y)\cdot\xi+\varphi(u,\xi,t_0)\right]\right|\gtrsim |x-y|
\end{align*}
if $|x-y|\geq \varpi$. Therefore, via an integration by parts, we can bound the kernel of $\widetilde{M}^1_{u,\Gamma_j,k}$, i.e.,
\begin{align*}
K_{j,k}(x,u,y):=\int_{\mathbb{R}^{2}} e^{i(x-y)\cdot\xi}e^{i \varphi(u,\xi,t_0)}m_{j,k}(u,\xi) \,\textrm{d}\xi,
\end{align*}
by $2^{(-\frac{1}{2}-N)k}\frac{1}{|x-y|^N}$ for some $N\in \mathbb{N}$ large enough if  $|x-y|\geq \varpi$. Furthermore, we obtain that
\begin{align}\label{eq:2.101}
\left\|\widetilde{M}^{1,a}_{u,\Gamma_j,k}f\right\|_{L^{q}(\mathbb{R}^{2})} \lesssim 2^{(-\frac{1}{2}-N)k} \|f\|_{L^{p}(\mathbb{R}^{2})}
\end{align}
for all $q\geq p \geq 1$, where
$$\widetilde{M}^{1,a}_{u,\Gamma_j,k}f(x):=\int_{\mathbb{R}^{2}} \chi_{B^{\complement}(x,\varpi)} (y)K_{j,k}(x,u,y) f(y) \,\textrm{d}y.$$
The proof of the $L^p(\mathbb{R}^2)\rightarrow L^q(\mathbb{R}^2)$ boundedness of $\partial_u(\widetilde{M}^{1,a}_{u,\Gamma_j,k}f)$ is similar as that of $\widetilde{M}^{1,a}_{u,\Gamma_j,k}f$ with some slight modifications. First write
\begin{align*}
\partial_u\left( e^{i \varphi(u,\xi,t_0)}m_{j,k}(u,\xi)\right) =& e^{i \varphi(u,\xi,t_0)}(-\xi_1t_0-u\xi_2\Gamma_j(t_0) )m_{j,k}(u,\xi)\\
&+ e^{i \varphi(u,\xi,t_0)}\chi\left(\frac{|\xi_1|}{|\xi_2|}\right)\psi'_k(u|\xi|) 2^{-k}|\xi| \int_{\mathbb{R}}  e^{-iu\xi_2t^2\eta(t,t_0) }\psi(t+t_0)\,\textrm{d}t\\
&+ e^{i \varphi(u,\xi,t_0)}\chi\left(\frac{|\xi_1|}{|\xi_2|}\right)\psi_k(u|\xi|) \int_{\mathbb{R}}  e^{-iu\xi_2t^2\eta(t,t_0) }(-i\xi_2t^2\eta(t,t_0) )\psi(t+t_0)\,\textrm{d}t
\end{align*}
and notice that $|\xi_1|\approx |\xi_2|\approx 2^k$, $u\approx1$, $t_0\approx 1$ and $\Gamma_j(t_0)\approx 1$, it is easy to see that the properties of $\partial_u( e^{i \varphi(u,\xi,t_0)}m_{j,k}(u,\xi))$ just like $2^k e^{i \varphi(u,\xi,t_0)}m_{j,k}(u,\xi)$. Therefore, we may obtain
\begin{align}\label{eq:2.102}
\left\|\partial_u\left(\widetilde{M}^{1,a}_{u,\Gamma_j,k}f\right)\right\|_{L^{q}(\mathbb{R}^{2})} \lesssim 2^k2^{(-\frac{1}{2}-N)k} \|f\|_{L^{p}(\mathbb{R}^{2})}
\end{align}
for all $q\geq p \geq 1$. From \eqref{eq:2.101}, the $L^p(\mathbb{R}^2)\rightarrow L^q(\mathbb{R}^2)$ boundedness of $\widetilde{M}^{1,a}_{1,\Gamma_j,k}$ with a upper bound $2^{(-\frac{1}{2}-N)k}$ can also been obtained. This is because our estimate \eqref{eq:2.101} is independent of $u\in[1,2]$. Consequently, as in the treatment of \eqref{eq:2.6}, we can establish the $L^{p}(\mathbb{R}^{2})\rightarrow L^{q}(\mathbb{R}^{2})$ boundedness of $\sup_{u\in [1,2]}| \widetilde{M}^{1,a}_{u,\Gamma_j,k}|$ with a upper bound $2^{(-\frac{1}{2}-N)k }+ 2^{(-\frac{1}{2}-N)k }2^{\frac{k}{q}}$. Furthermore, combining Lemma \ref{lemma 2.10} and the fact that $N\in \mathbb{N}$ large enough, it is easy to obtain the desired estimate \eqref{eq:2.11} for $\sup_{u\in [1,2]}| \widetilde{M}^{1,a}_{u,\Gamma_j,k}|$.

It remains to estimate the operator
$$\widetilde{M}^{1,b}_{u,\Gamma_j,k}f(x):=\int_{\mathbb{R}^{2}} \chi_{B(x,\varpi)} (y)K_{j,k}(x,u,y) f(y) \,\textrm{d}y.$$
For this purpose, we begin with some definitions. Let $\Omega:\ \mathbb{R}^2\times\mathbb{R}\rightarrow\mathbb{R}$ be a nonnegative smooth function, identically equal to one on $\{(x,u)\in \mathbb{R}^2\times\mathbb{R}:\ |x|\leq \varpi, u\in [1,2]\}$ and vanishing outside $\{(x,u)\in \mathbb{R}^2\times\mathbb{R}:\  |x|\leq 2\varpi, u\in(\frac{1}{2}, \frac{5}{2})\}$, and
$$\Psi(x,u,\xi):=x\cdot\xi+\varphi(u,\xi,t_0),~~~ a_{j,k}(x,u,\xi):=\Omega(x,u)m_{j,k}(u,\xi).$$ We define
\begin{align}\label{eq:2.22}
A_{j,k}f(x,u):=\int_{\mathbb{R}^{2}} e^{i\Psi(x,u,\xi)}a_{j,k}(x,u,\xi) \hat{f}(\xi)\,\textrm{d}\xi,
\end{align}
which is a localized version of $\widetilde{M}^1_{u,\Gamma_j,k}$. The operator $A_{j,k}$ is related to a class of Fourier integral operators studied in many papers (see, for instance, \cite{MSS, B, Lee}), and the Fourier integral operators are originated from the study of pseudo-differential operators or half-wave propagator.

We show that it suffices to obtain the following estimate: there exists a constant $\mu\in (\frac{1}{q}, \infty)$ such that
\begin{align}\label{eq:2.103}
\left\|A_{j,k}f\right\|_{L^{q}(\mathbb{R}^{3})} \lesssim 2^{- \mu k} \|f\|_{L^{p}(\mathbb{R}^{2})}
\end{align}
for all $(\frac{1}{p},\frac{1}{q})\in \Delta$, where the implicit constant is independent of $j$ and $k$. The proof of \eqref{eq:2.103} is the key to this paper, the constant $\mu\in (\frac{1}{q}, \infty)$ plays an important role in the following \eqref{eq:2.h}, which ensures that the series $\sum_{k\geq 1} 2^{-\mu k }2^{\frac{k}{q}} $ converges and further, by Lemma \ref{lemma 2.10}, implies the desired estimate \eqref{eq:2.11} for $\sup_{u\in [1,2]} |\widetilde{M}^{1,b}_{u,\Gamma_j,k}|$.

Indeed, we decompose $\mathbb{R}^{2}$ as $\bigcup_{i\in \mathbb{Z}^{2}}B(x_i,\varpi)$ such that for any $i\neq i'$, $|x_i-x_{i'}|\approx |i-i'|\varpi$. Then, it is easy to see that\footnote{We note $\|\widetilde{M}^{1,b}_{u,\Gamma_j,k}f\|_{L^{q}(\mathbb{R}^{2}\times[1,2])}:= [\int_1^2\int_{\mathbb{R}^{2}} |\widetilde{M}^{1,b}_{u,\Gamma_j,k}f(x)|^q \,\textrm{d}x \,\textrm{d}u]^{\frac{1}{q}}$.}
\begin{align*}
\left\|\widetilde{M}^{1,b}_{u,\Gamma_j,k}f\right\|^q_{L^{q}(\mathbb{R}^{2}\times[1,2])}\lesssim\sum_{i\in \mathbb{Z}^{2}}\left\|\int_{\mathbb{R}^{2}} \chi_{B(x,\varpi)} (y)K_{j,k}(x,u,y) f(y) \,\textrm{d}y\right\|^q_{L^{q}(B(x_i,\varpi)\times [1,2] )} .
\end{align*}
By changing of variables, together with the fact that $K_{j,k}(x-x_i,u,y-x_i)=K_{j,k}(x,u,y)$, we can write the last display as
\begin{align*}
\sum_{i\in \mathbb{Z}^{2}}\left\|\int_{\mathbb{R}^{2}} \chi_{B(x,\varpi)} (y)K_{j,k}(x,u,y) f(y-x_i) \,\textrm{d}y\right\|^q_{L^{q}(B(0,\varpi)\times [1,2] )}.
\end{align*}
Let
\begin{align*}
B_{j,k}f(x,u):=\Omega(x,u)\int_{\mathbb{R}^{2}} \chi_{B(x,\varpi)} (y)K_{j,k}(x,u,y) f(y) \,\textrm{d}y,
\end{align*}
we then have
\begin{align}\label{eq:2.a}
\left\|\widetilde{M}^{1,b}_{u,\Gamma_j,k}f\right\|^q_{L^{q}(\mathbb{R}^{2}\times[1,2])}
&\lesssim\sum_{i\in \mathbb{Z}^{2}}\left\|B_{j,k}(f(\cdot-x_i))\right\|^q_{L^{q}(\mathbb{R}^{3} )}\\
&\lesssim \sum_{i\in \mathbb{Z}^{2}}\left\|(B_{j,k}-A_{j,k})(f(\cdot-x_i))\right\|^q_{L^{q}(\mathbb{R}^{3} )}+\sum_{i\in \mathbb{Z}^{2}}\left\|A_{j,k}(f(\cdot-x_i))\right\|^q_{L^{q}(\mathbb{R}^{3} )}.\nonumber
\end{align}
We remark that the left hand side of \eqref{eq:2.a} is a $L^q$ norm on $\mathbb{R}^{2}\times [1,2]$ and previous estimates (see, for example, \eqref{eq:2.10101}, \eqref{eq:2.101} and \eqref{eq:2.102}) are all a $L^q$ norm on $\mathbb{R}^{2}$. This is because we need a local smoothing estimate to obtain \eqref{eq:2.103}, and the integral over $[1,2]$ about $u$ plays an important role in this process.

The operator $B_{j,k}-A_{j,k}$ can be handled in a way similar to $\widetilde{M}^{1,a}_{u,\Gamma_j,k}$ since the kernel of $B_{j,k}-A_{j,k}$ is supported on $\{(x,y)\in\mathbb{R}^{2}:\ |x-y|\geq \varpi\}$. Therefore, as in the treatment of \eqref{eq:2.101}, it is not difficult to obtain
\begin{align}\label{eq:2.b}
\sum_{i\in \mathbb{Z}^{2}}\left\|(B_{j,k}-A_{j,k})(f(\cdot-x_i))\right\|^q_{L^{q}(\mathbb{R}^{3} )}\lesssim & \sum_{i\in \mathbb{Z}^{2}}2^{(-\frac{1}{2}-N)k q} \|f(\cdot-x_i)  \chi_{B(0,2\varpi)} (\cdot)\|^q_{L^{p}(\mathbb{R}^{2})}\\
\lesssim &2^{(-\frac{1}{2}-N)k q} \|f\|^q_{L^{p}(\mathbb{R}^{2})}\nonumber
\end{align}
for some $N\in \mathbb{N}$ large enough. As in \eqref{eq:2.102}, we also have
\begin{align}\label{eq:2.c}
\sum_{i\in \mathbb{Z}^{2}}\left\|\partial_u\left((B_{j,k}-A_{j,k})(f(\cdot-x_i))\right)\right\|^q_{L^{q}(\mathbb{R}^{3})}\lesssim 2^{kq}2^{(-\frac{1}{2}-N)k q} \|f\|^q_{L^{p}(\mathbb{R}^{2})}.
\end{align}
On the other hand, it follows from \eqref{eq:2.103} that
\begin{align}\label{eq:2.d}
\sum_{i\in \mathbb{Z}^{2}}\left\|A_{j,k}(f(\cdot-x_i))\right\|^q_{L^{q}(\mathbb{R}^{3} )}\lesssim \sum_{i\in \mathbb{Z}^{2}}2^{-\mu k q}\|f(\cdot-x_i)  \chi_{B(0,2\varpi)} (\cdot)\|^q_{L^{p}(\mathbb{R}^{2})}\lesssim 2^{-\mu k q} \|f\|^q_{L^{p}(\mathbb{R}^{2})}.
\end{align}
From \eqref{eq:2.103}, as in the treatment of \eqref{eq:2.102}, one may get
\begin{align}\label{eq:2.e}
\sum_{i\in \mathbb{Z}^{2}}\left\|\partial_u\left(A_{j,k}(f(\cdot-x_i))\right)\right\|^q_{L^{q}(\mathbb{R}^{3} )}\lesssim 2^{kq} 2^{-\mu k q} \|f\|^q_{L^{p}(\mathbb{R}^{2})}.
\end{align}

Based on these estimates we are invited to bound $\sup_{u\in [1,2]} |\widetilde{M}^{1,b}_{u,\Gamma_j,k}|$. From \eqref{eq:2.a}, combining \eqref{eq:2.b} and \eqref{eq:2.d}, we establish that
\begin{align}\label{eq:2.f}
\left\|\widetilde{M}^{1,b}_{u,\Gamma_j,k}f\right\|_{L^{q}(\mathbb{R}^{2}\times[1,2])}\lesssim \left(2^{(-\frac{1}{2}-N)k }+ 2^{-\mu k } \right)\|f\|_{L^{p}(\mathbb{R}^{2})}.
\end{align}
As in \eqref{eq:2.a}, combining \eqref{eq:2.c} and \eqref{eq:2.e}, we have that
\begin{align}\label{eq:2.g}
\left\|\partial_u\left(\widetilde{M}^{1,b}_{u,\Gamma_j,k}f\right)\right\|_{L^{q}(\mathbb{R}^{2}\times[1,2])}\lesssim 2^k\left(2^{(-\frac{1}{2}-N)k }+ 2^{-\mu k } \right)\|f\|_{L^{p}(\mathbb{R}^{2})}.
\end{align}
It is easy to obtain the $L^p(\mathbb{R}^2)\rightarrow L^q(\mathbb{R}^2)$ boundedness of $\widetilde{M}^{1,b}_{1,\Gamma_j,k}$ with a upper bound $2^{(-\frac{1}{2}-N)k }+ 2^{-\mu k } $. Therefore, as in the treatment of \eqref{eq:2.99}, by the estimates \eqref{eq:2.f} and \eqref{eq:2.g}, we conclude that
\begin{align}\label{eq:2.h}
\left\|\sup_{u\in [1,2]} \left|\widetilde{M}^{1,b}_{u,\Gamma_j,k}f\right|\right\|_{L^{q}(\mathbb{R}^{2})}\lesssim \left[2^{(-\frac{1}{2}-N)k }+ 2^{-\mu k }+ \left(2^{(-\frac{1}{2}-N)k }+ 2^{-\mu k }\right)2^{\frac{k}{q}}\right] \|f\|_{L^{p}(\mathbb{R}^{2})}
\end{align}
for all $(\frac{1}{p},\frac{1}{q})\in \Delta$, where the implicit constant is independent of $j$ and $k$. This, combined with Lemma \ref{lemma 2.10} and the facts that $\mu\in (\frac{1}{q}, \infty)$ and $N\in \mathbb{N}$ large enough, leads to the desired estimate \eqref{eq:2.11} for $\sup_{u\in [1,2]} |\widetilde{M}^{1,b}_{u,\Gamma_j,k}|$.

Thus, we reduce the problem to proving \eqref{eq:2.103}, we will use the well-known local smoothing estimates. Our estimate \eqref{eq:2.103} is based on the following two lemmas obtained in \cite{B} and \cite{Lee}, respectively. Beltran, Hickman and Sogge \cite{B} obtained $L^p(\mathbb{R}^n)\rightarrow L^p(\mathbb{R}^{n+1})$ boundedness for a class of Fourier integral operators satisfying the curvature condition in \cite{B} and it is equivalent to the cinematic curvature condition defined in \cite{Sog}.

\begin{lemma}\label{lemma 2.5}
(\cite[Proposition 3.2]{B}) Let
$$T_{\nu}f(z):=\int_{\mathbb{R}^{n}} e^{i\phi(z,\xi)}a(z,\xi) \frac{\hat{f}(\xi)}{(1+|\xi|^2)^{\frac{\nu}{2}}}\,\textrm{d}\xi,$$
where $a(z,\xi)$ is a symbol of order zero. Here and below $z$ is used to denote vector in $\mathbb{R}^{n}\times \mathbb{R}$ comprised of the space-time variables $(x,u)$. Suppose $\textrm{supp}~a(\cdot,\xi)$ is contained in a fixed compact set and suppose that $\phi(z,\cdot)$ is a homogeneous function of degree one. For all $(z,\xi)\in \textrm{supp}~a$, $\phi$ satisfies:
\begin{enumerate}
  \item[\rm(i)]  $\textrm{rank}~\partial^2_{z\xi}\phi=n$;
  \item[\rm(ii)] $\textrm{rank}~\partial^2_{\xi\xi}\langle\partial_z\phi,\theta\rangle=n-1$ provided $\theta\in \mathbb{S}^n$ is the direction (unique up to sign) for which $\nabla_\xi\langle\partial_z\phi,\theta\rangle=0$.
\end{enumerate}
Then for $\frac{2(n+1)}{n-1}\leq p<\infty$,
$$\left\|T_{\nu}f\right\|_{L^{p}(\mathbb{R}^{n+1})} \lesssim \|f\|_{L^{p}(\mathbb{R}^{n})}$$
provided $\nu> (n-1)|\frac{1}{2}-\frac{1}{p}|-\frac{1}{p}$.
\end{lemma}

On the other hand, Lee \cite{Lee} used the bilinear method to establish $L^p(\mathbb{R}^n)\rightarrow L^q(\mathbb{R}^{n+1})$ estimates for a class of Fourier integral operators satisfying the so called cinematic curvature condition and an additional condition, which gives an improvement of the $L^2(\mathbb{R}^n)\rightarrow L^q(\mathbb{R}^{n+1})$ boundedness in \cite{MSS}.

\begin{lemma}\label{lemma 2.4}
(\cite[Corollary 1.5]{Lee}) Let $T_{\nu}$ be defined as above. For all $(z,\xi)\in \textrm{supp}~a$, $\phi$ satisfies:
\begin{enumerate}
  \item[\rm(i)]  $\textrm{rank}~\partial^2_{z\xi}\phi=n$;
  \item[\rm(ii)] $\textrm{rank}~\partial^2_{\xi\xi}\langle\partial_z\phi,\theta\rangle=n-1$ provided $\theta\in \mathbb{S}^n$ is the direction (unique up to sign) for which $\nabla_\xi\langle\partial_z\phi,\theta\rangle=0$;
  \item[\rm(iii)] also all nonzero eigenvalues of $\partial^2_{\xi\xi}\langle\partial_z\phi,\theta\rangle$ have the same sign.
\end{enumerate}
Then for $\frac{2(n^2+2n-1)}{n^2-1}\leq q \leq \infty$, $\frac{n+1}{q}\leq (n-1)(1-\frac{1}{p})$ and $q\geq \frac{n+3}{n+1}p$,
$$\left\|T_{\nu}f\right\|_{L^{q}(\mathbb{R}^{n+1})} \lesssim \|f\|_{L^{p}(\mathbb{R}^{n})}$$
provided $\nu> \frac{1}{p}-\frac{n+1}{q}+\frac{n-1}{2}$.
\end{lemma}

The rest part of this section is devoted to a proof of \eqref{eq:2.103} based on Lemmas \ref{lemma 2.5} and \ref{lemma 2.4} in our case $n=2$. We first give the following lemma.

\begin{lemma}\label{lemma 2.3} Recall that $\eta(t,t_0)=\frac{\Gamma''_j(t_0)}{2}+\frac{t}{2}\int_0^1(1-\theta)^2\Gamma'''_j(\theta t+t_0)\,\textrm{d}\theta$. Let us set
\begin{align*}
h(\lambda,t_0):=\int_{\mathbb{R}}  e^{-i\lambda t^2\eta(t,t_0) }\psi(t+t_0)\,\textrm{d}t.
\end{align*}
Then, for any $|\lambda|\geq 1$ and $t_0\in(\frac{1}{2},2)$, one has
\begin{align*}
\left|\partial_{\lambda}^{\alpha}\partial_{t_{0}}^{\beta}h(\lambda,t_{0})\right|\lesssim |\lambda|^{-\frac{1}{2}-\alpha}
\end{align*}
for all $\alpha,\beta\in \mathbb{N}$ and $\beta<N-2$, where the implicit constant is independent of $j$.
\end{lemma}

\begin{proof}[Proof (of Lemma \ref{lemma 2.3}).] We first show that $|\partial_{t_{0}}^{\beta}h(\lambda,t_{0})|\lesssim |\lambda|^{-\frac{1}{2}}$ for all $\beta<N-2$. By the Leibniz rule and the Fa\`{a} di Bruno formula (see, for instance, \cite{Jo}), we can write $\partial_{t_{0}}^{\beta}h(\lambda,t_{0})$ as
\begin{align}\label{eq:2.14}
\sum_{m+n=\beta}\frac{\beta!}{m!n!} \int_{\mathbb{R}}  e^{-i\lambda t^2\eta(t,t_0) } \sum_{\nabla} & \frac{n!}{b_1 !\cdots b_n !} (-i\lambda t^2 )^{\iota} \\
&\times\left(\frac{\partial_{t_{0}}^{1}\eta(t,t_0) }{1!}\right)^{b_1} \cdots \left(\frac{\partial_{t_{0}}^{n}\eta(t,t_0) }{n!}\right)^{b_n} \partial_{t_{0}}^{m}\psi(t+t_0)\,\textrm{d}t\nonumber
\end{align}
for all $\beta<N-2$, where the sum $\sum_{\nabla}$ is over all different solutions in nonnegative integers $b_1,\cdots,b_n$ of $b_1+2b_2+\cdots+nb_n=n$, and $\iota:=b_1+\cdots+b_n$. Since the sum $\sum_{m+n=\beta}\sum_{\nabla}$ in \eqref{eq:2.14} contains only a finite number of indices, it suffices to obtain that
\begin{align}\label{eq:2.15}
\left| \int_{\mathbb{R}}  e^{-i\lambda t^2\eta(t,t_0) }  t^{2\iota} u(t,t_0)\,\textrm{d}t\right|\lesssim |\lambda|^{-\frac{1}{2}-\iota},
\end{align}
 where
 \begin{align}\label{eq:2.015}
 u(t,t_0):=\left(\frac{\partial_{t_{0}}^{1}\eta(t,t_0) }{1!}\right)^{b_1} \cdots \left(\frac{\partial_{t_{0}}^{n}\eta(t,t_0) }{n!}\right)^{b_n} \partial_{t_{0}}^{m}\psi(t+t_0)
 \end{align}
 is a compactly supported smooth function.

Let $I_{\iota}$ be the integral in \eqref{eq:2.15}. To estimate it, let $\vartheta$ be a nonnegative smooth function on $\mathbb{R}$, identically equal one on $\{t\in \mathbb{R}:\  |t|\leq 1\}$ and vanishing off $\{t\in \mathbb{R}:\  |t|\leq 2\}$, we shall break up $I_{\iota}$ into the following two parts:
\begin{align*}
 I_{\iota}  =  \int_{\mathbb{R}}  e^{-i\lambda t^2\eta(t,t_0) }  t^{2\iota} u(t,t_0)\vartheta\left(\frac{t}{\varepsilon}\right)\,\textrm{d}t+\int_{\mathbb{R}}  e^{-i\lambda t^2\eta(t,t_0) }  t^{2\iota} u(t,t_0)\left(1-\vartheta\left(\frac{t}{\varepsilon}\right)\right)\,\textrm{d}t
 =: I+II,
\end{align*}
where $\varepsilon>0$ will be determined momentarily.

The first integral $I$ is easy to handle. Notice that $t_0\in (\frac{1}{2},2)$, $t+t_0\in [\frac{1}{2},2]$ and $\theta \in (0,1)$, we may get $\theta t+t_0 \in [\frac{1}{2},2]$. This, combined with Lemma \ref{lemma 2.1}, implies that $|\Gamma_{j}^{(n+2)}(t_{0})|\lesssim 1$ and  $|\Gamma_{j}^{(n+3)}(\theta t+t_{0})|\lesssim 1$ for all $ n<N-2$. Consequently, we obtain $|u(t,t_0)| \lesssim 1$. Therefore, by taking absolute values, we have
\begin{align}\label{eq:2.16}
|I|\lesssim \left| \int_{|t|\leq 2 \varepsilon}    t^{2\iota} \,\textrm{d}t\right|\lesssim \varepsilon^{2\iota+1}.
\end{align}

It remains to estimate the second integral $II$, we shall need to integrate by parts. Let $$D^{*}f(t):=\frac{\textrm{d}}{\textrm{d}t}\left( \frac{f(t)}{ i\Gamma''_j(t_0) \lambda t }\right).$$
We can rewrite $II$ as
$$\int_{\mathbb{R}}  e^{-i\frac{\Gamma''_j(t_0)}{2}\lambda t^2 }  (D^{*})^{\kappa}  \left[t^{2\iota}  e^{-i\frac{\lambda}{2} t^3 \int_0^1(1-\theta)^2\Gamma'''_j(\theta t+t_0)\,\textrm{d}\theta} u(t,t_0)\left(1-\vartheta\left(\frac{t}{\varepsilon}\right)\right) \right]   \,\textrm{d}t$$
for all $\kappa\in \mathbb{N}$. Furthermore, by taking absolute values, enable us to bound $|II|$ by
$$\int_{|t|>\varepsilon}  \left|  (D^{*})^{\kappa}  \left[t^{2\iota}  e^{-i\frac{\lambda}{2} t^3 \int_0^1(1-\theta)^2\Gamma'''_j(\theta t+t_0)\,\textrm{d}\theta} u(t,t_0)\left(1-\vartheta\left(\frac{t}{\varepsilon}\right)\right) \right] \right|  \,\textrm{d}t.$$
On the other hand, from $(\textrm{iii})$ of Lemma \ref{lemma 2.1}, we have $|\Gamma''_j(t_0)|\gtrsim 1$. By this observation together with the product rule for differentiation, we have that
\begin{align}\label{eq:2.17}
|II|\lesssim \lambda^{-\kappa}\int_{|t|\geq\varepsilon} |t|^{2\iota-2\kappa}\,\textrm{d}t\lesssim\lambda^{-\kappa}\varepsilon^{2\iota-2\kappa+1}
\end{align}
if $2\iota-2\kappa+1<0$.

Putting together our estimates \eqref{eq:2.16} and \eqref{eq:2.17}, we may obtain
\begin{align}\label{eq:2.18}
|I_{\iota}|\lesssim \varepsilon^{2\iota+1}+\lambda^{-\kappa}\varepsilon^{2\iota-2\kappa+1}
\end{align}
if $\kappa$ satisfying $2\iota-2\kappa+1<0$. It is easy to see that the right side of \eqref{eq:2.18} is smallest when the two summands agree, i.e., $\varepsilon:=\lambda^{-\frac{1}{2}}$, which gives
$$|I_\iota|\lesssim |\lambda|^{-\frac{1}{2}-\iota}.$$
This is \eqref{eq:2.15} as desired.

We now turn to $|\partial_{\lambda}^{\alpha}\partial_{t_{0}}^{\beta}h(\lambda,t_{0})|\lesssim |\lambda|^{-\frac{1}{2}-\alpha}$ for all $\alpha,\beta\in \mathbb{N}$ and $\beta<N-2$. In fact, by the Leibniz rule, it follows that
$$\partial_{\lambda}^{\alpha}\left(e^{-i\lambda t^2\eta(t,t_0) } \lambda^{\iota} \right)= \sum_{w+v=\alpha}\frac{\alpha!}{w!v!}\frac{{\iota}!}{(\iota-w)!} e^{-i\lambda t^2\eta(t,t_0) } (-i t^2\eta(t,t_0))^v \lambda^{\iota-w}$$
if $w$ satisfying $w\leq \iota$, and zero otherwise. This, combined with \eqref{eq:2.14} and \eqref{eq:2.015},  we can then write $\partial_{\lambda}^{\alpha}\partial_{t_{0}}^{\beta}h(\lambda,t_{0})$ as
\begin{align}\label{eq:2.20}
\sum_{m+n=\beta}\frac{\beta!}{m!n!} \sum_{\nabla} \frac{n!}{b_1 !\cdots b_n !}\sum_{w+v=\alpha}&\frac{\alpha!}{w!v!}\frac{\iota!}{(\iota-w)!} \\
&\times\lambda^{\iota-w}  \int_{\mathbb{R}}  e^{-i\lambda t^2\eta(t,t_0) }t^{2v+2\iota} (-i\eta(t,t_0))^v  (-i )^{\iota} u(t,t_0)\,\textrm{d}t\nonumber
\end{align}
for all $\alpha,\beta\in \mathbb{N}$ and $\beta<N-2$. As in the treatment of \eqref{eq:2.15}, we claim that the following holds:
\begin{align}\label{eq:2.21}
\left| \int_{\mathbb{R}}  e^{-i\lambda t^2\eta(t,t_0) }t^{2v+2\iota} (-i\eta(t,t_0))^v  (-i )^{\iota} u(t,t_0)\,\textrm{d}t\right|\lesssim |\lambda|^{-\frac{1}{2}-v-\iota}.
\end{align}
Noting that the sum $\sum_{m+n=\beta}\sum_{\nabla}\sum_{w+v=\alpha}$ in \eqref{eq:2.20} contains only a finite number of indices and $w+v=\alpha$. This, combined with \eqref{eq:2.21}, yields
\begin{align*}
\left|\partial_{\lambda}^{\alpha}\partial_{t_{0}}^{\beta}h(\lambda,t_{0})\right|\lesssim |\lambda|^{-\frac{1}{2}-\alpha}
\end{align*}
for all $\alpha,\beta\in \mathbb{N}$ and $\beta<N-2$ as required. Hence, we finish the proof of Lemma \ref{lemma 2.3}.
\end{proof}

Now, for $A_{j,k}$ defined in \eqref{eq:2.22}, we continue the estimate by showing that $a_{j,k}(x,u,\xi)$ is a symbol of order $-\frac{1}{2}$ by Lemma \ref{lemma 2.3}. For $t_0=(\Gamma'_j)^{-1}(-\frac{\xi_1}{\xi_2})$, notice that $|\xi|\approx 2^{k}\geq 1$ and $\frac{|\xi_1|}{|\xi_2|}\approx 1$, by this observation together with $(\textrm{v})$ of Lemma \ref{lemma 2.1} yields
\begin{align*}
\left|\partial_{\xi}^{\alpha}t_{0}\right|\lesssim (1+|\xi|)^{-|\alpha|}
\end{align*}
for all $\alpha\in \mathbb{N}_0^2$ with $|\alpha|<N$. This, combined with Lemma \ref{lemma 2.3}, implies that
\begin{align}\label{eq:2.23}
\left|\partial_{u}^{\beta}\partial_{\xi}^{\alpha}m_{j,k}(u,\xi)\right|\lesssim (1+|\xi|)^{-\frac{1}{2}-|\alpha|}
\end{align}
for all $(\alpha,\beta)\in \mathbb{N}^2_0\times \mathbb{N}_0$ with $|\alpha|+|\beta|<N$. Consequently,
\begin{align}\label{eq:2.24}
\left|\partial_{x}^{\gamma}\partial_{u}^{\beta}\partial_{\xi}^{\alpha}a_{j,k}(x,u,\xi)\right|
\leq(1+|\xi|)^{-\frac{1}{2}-|\alpha|}
\end{align}
for all $(\alpha,\beta,\gamma)\in \mathbb{N}^2_0\times \mathbb{N}_0\times \mathbb{N}^2_0$ with $|\alpha|+|\beta|+|\gamma|<N$, where the implicit constant is independent of $j$ and $k$. Therefore, $a_{j,k}(x,u,\xi)$ is a symbol of order $-\frac{1}{2}$. By simple calculation, we may obtain that $a_{j,k}(x,u,\xi)(1+|\xi|)^{\frac{1}{2}}$ is a symbol of order zero. We also note that $\textrm{supp} ~a_{j,k} \subset \{x\in \mathbb{R}^2:\ |x|\leq 2\varpi\}\times[\frac{1}{2},\frac{5}{2}]\times\{\xi\in \mathbb{R}^2:\ |\xi_1|\approx|\xi_2|\approx2^{k}\}$ and $\Psi(x,u,\cdot)$ is a homogeneous function of degree one.

We turn to verify the conditions $\textrm{(i)}, \textrm{(ii)}$ and $\textrm{(iii)}$ in Lemmas \ref{lemma 2.5} and \ref{lemma 2.4}. Recall that $\Psi(x,u,\xi)=x\cdot\xi+\varphi(u,\xi,t_0)=x_1\xi_1+x_2\xi_2-u\xi_1 t_0- u \xi_2 \Gamma_j(t_0)$ and $t_0=(\Gamma'_j)^{-1}(-\frac{\xi_1}{\xi_2})$. From
\begin{align}\label{eq:2.25}
\partial^2_{z\xi}\Psi=\left(
\begin{array}{ccc}
\partial^2_{x_1 \xi_1} \Psi~& ~\partial^2_{x_2 \xi_1} \Psi ~& ~\partial^2_{u\xi_1} \Psi  \\
\partial^2_{x_1 \xi_2} \Psi~& ~\partial^2_{x_2 \xi_2} \Psi ~& ~\partial^2_{u\xi_2} \Psi
\end{array}
\right)=\left(
\begin{array}{ccc}
1~& ~0 ~& ~-t_0 \\
0~& ~1 ~& ~-  \Gamma_j(t_0)
\end{array}
\right),
\end{align}
It is easy to see that the rank of the mixed Hessian of $\Psi$ is $2$, this is the condition $(\textrm{i})$ of Lemma \ref{lemma 2.5} or $(\textrm{i})$ of Lemma \ref{lemma 2.4} as desired. For the condition $\textrm{(ii)}$ in Lemmas \ref{lemma 2.5} and \ref{lemma 2.4}, we first write $$\partial_z \Psi=(\partial_{x_1} \Psi, \partial_{x_2} \Psi, \partial_u\Psi)=(\xi_1, \xi_2, -\xi_1 t_0-  \xi_2 \Gamma_j(t_0)).$$
Let $\theta:=(\theta_1,\theta_2,\theta_3)\in \mathbb{S}^3$ be the direction (unique up to sign) satisfying $\nabla_\xi\langle\partial_z\Psi,\theta\rangle=0$, combining the fact that $\Gamma'_j(t_0)=-\frac{\xi_1}{\xi_2}$, we then have
\begin{align*}
\langle\partial_z\Psi,\theta\rangle=\xi_1\theta_1+\xi_2\theta_2-\xi_1\theta_3t_0-\xi_2\theta_3\Gamma_j(t_0)
\end{align*}
and
\begin{eqnarray*}
\left\{\aligned
&\theta_1-\theta_3t_0=0 ;\\
&\theta_2 -\theta_3\Gamma_j(t_0)=0.
\endaligned\right.
\end{eqnarray*}
Therefore,
\begin{align*}
\theta=\theta(\xi)=\left(\frac{t_0}{\sqrt{1+t_0^2+\Gamma_j(t_0)^2}}, \frac{\Gamma_j(t_0)}{\sqrt{1+t_0^2+\Gamma_j(t_0)^2}}, \frac{1}{\sqrt{1+t_0^2+\Gamma_j(t_0)^2}}\right).
\end{align*}
Furthermore,
\begin{align*}
\partial^2_{\xi\xi}\langle\partial_z\Psi,\theta(\eta)\rangle~|_{~\xi=\eta}=\left(
\begin{array}{ccc}
\partial^2_{\xi_1 \xi_1} \langle\partial_z\Psi,\theta(\eta)\rangle~& ~\partial^2_{ \xi_1\xi_2}\langle\partial_z\Psi,\theta(\eta)\rangle \\
\partial^2_{\xi_2\xi_1}\langle\partial_z\Psi,\theta(\eta)\rangle~& ~\partial^2_{\xi_2 \xi_2} \langle\partial_z\Psi,\theta(\eta)\rangle
\end{array}
\right)~\bigg{|}_{~\xi=\eta}.
\end{align*}
By calculation, it now follows that
\begin{eqnarray*}
\left\{\aligned
&\partial^2_{\xi_1 \xi_1} \langle\partial_z\Psi,\theta(\eta)\rangle=-\theta_3(\eta)\partial_{\xi_1}t_0 ;\\
&\partial^2_{ \xi_1\xi_2}\langle\partial_z\Psi,\theta(\eta)\rangle=-\theta_3(\eta)\partial_{\xi_2}t_0 ;\\
&\partial^2_{\xi_2\xi_1}\langle\partial_z\Psi,\theta(\eta)\rangle= -\theta_3(\eta)\Gamma'_j(t_0) \partial_{\xi_1}t_0;\\
&\partial^2_{\xi_2 \xi_2} \langle\partial_z\Psi,\theta(\eta)\rangle=-\theta_3(\eta)\Gamma'_j(t_0) \partial_{\xi_2}t_0,
\endaligned\right.
\end{eqnarray*}
which gives the desired identity
\begin{align}\label{eq:2.26}
\partial^2_{\xi\xi}\langle\partial_z\Psi,\theta(\eta)\rangle~|_{~\xi=\eta}=\left(
\begin{array}{ccc}
-\theta_3(\eta)\partial_{\xi_1}t_0~& ~-\theta_3(\eta)\partial_{\xi_2}t_0 \\
-\theta_3(\eta)\Gamma'_j(t_0) \partial_{\xi_1}t_0~& ~-\theta_3(\eta)\Gamma'_j(t_0) \partial_{\xi_2}t_0
\end{array}
\right).
\end{align}
Notice that $\partial_{\xi_1}t_0=-\frac{1}{\xi_2}\frac{1}{\Gamma''_j(t_0)}$ and $\partial_{\xi_2}t_0=\frac{\xi_1}{\xi_2^2}\frac{1}{\Gamma''_j(t_0)}$, one obtains $\partial^2_{ \xi_1\xi_2}\langle\partial_z\Psi,\theta(\eta)\rangle=\partial^2_{\xi_2\xi_1}\langle\partial_z\Psi,\theta(\eta)\rangle$.
Furthermore, it is easy to see that
$$\textrm{rank}~\partial^2_{\xi\xi}\langle\partial_z\Psi,\theta(\eta)\rangle~|_{~\xi=\eta}=1,$$ this is the condition $(\textrm{ii})$ of Lemma \ref{lemma 2.5} or $(\textrm{i})$ of Lemma \ref{lemma 2.4} as desired. From \eqref{eq:2.26}, it is clear that there is only one nonzero eigenvalue of $\partial^2_{\xi\xi}\langle\partial_z\Psi,\theta(\eta)\rangle~|_{~\xi=\eta}$, which trivially leads to the condition $\textrm{(iii)}$ of Lemma \ref{lemma 2.4}.

Therefore, for any $ p\in[6,\infty)$ and $\delta\in(0,\infty)$, by Lemma \ref{lemma 2.5}, we have
\begin{align}\label{eq:2.27}
\|A_{j,k}f\|_{L^{p}(\mathbb{R}^{3})}= 2^{\left(-\frac{2}{p}+\delta\right) k}\left\|2^{\left(\frac{2}{p}-\delta\right) k}A_{j,k}f\right\|_{L^{p}(\mathbb{R}^{3})} \lesssim 2^{\left(-\frac{2}{p}+\delta\right) k}\|f\|_{L^{p}(\mathbb{R}^{2})}
\end{align}
with a bound independent of $j$. On the other hand, from \eqref{eq:2.23}, it is easy to obtain that the multiplier of the Fourier integral operators $A_{j,k}$ in \eqref{eq:2.22} can be bounded from above by $2^{-\frac{k}{2}}$. Hence, by Plancherel's theorem, one has
\begin{align}\label{eq:2.28}
\|A_{j,k}f\|_{L^{2}(\mathbb{R}^{3})}\lesssim 2^{-\frac{k}{2}}\|f\|_{L^{2}(\mathbb{R}^{2})}.
\end{align}
By interpolation, for any $ p\in[2,6]$ and $\delta\in(0,\infty)$, we get
\begin{align}\label{eq:2.29}
\|A_{j,k}f\|_{L^{p}(\mathbb{R}^{3})}\lesssim 2^{\left(-\frac{1}{4}-\frac{1}{2p}+3(\frac{1}{2}-\frac{1}{p})\delta\right)k}\|f\|_{L^{p}(\mathbb{R}^{2})},
\end{align}
where the implicit constant is independent of $j$ and $k$. By Lemma \ref{lemma 2.4}, for any $(\frac{1}{p}, \frac{1}{q})\in \{(\frac{1}{p}, \frac{1}{q}) :\ 0\leq \frac{1}{q} \leq \frac{3}{14}, \frac{3}{q}\leq 1-\frac{1}{p}, \frac{1}{q}\leq  \frac{3}{5p} \}$ and $\delta\in(0,\infty)$, one deduces that
\begin{align}\label{eq:2.30}
\|A_{j,k}f\|_{L^{q}(\mathbb{R}^{3})}= 2^{\left(\frac{1}{p}-\frac{3}{q}+\delta\right) k}\left\|2^{\left(-\frac{1}{p}+\frac{3}{q}-\delta\right) k}A_{j,k}f\right\|_{L^{q}(\mathbb{R}^{3})} \lesssim 2^{\left(\frac{1}{p}-\frac{3}{q}+\delta\right) k}\|f\|_{L^{p}(\mathbb{R}^{2})}
\end{align}
with a bound independent of $j$.

Observe that for any $(\frac{1}{p}, \frac{1}{q})\in \{(\frac{1}{p}, \frac{1}{q}) :\ 0\leq \frac{1}{q} \leq \frac{3}{14}, \frac{3}{q}\leq 1-\frac{1}{p}, \frac{1}{q}\leq  \frac{3}{5p} \}$, by setting $\delta:=\frac{1}{q}-\frac{1}{2p}$ in \eqref{eq:2.30}, we deduce that \eqref{eq:2.103} is established with $\mu:=\frac{2}{q}-\frac{1}{2p}$. Here, we must require $\frac{1}{2p}<\frac{1}{q}$, which produces one of the boundary of $\Delta$ and also one of the vertex $(\frac{2}{5},\frac{1}{5})$ of $\Delta$. The another boundary of $\Delta$, i.e., the dotted line connecting $(\frac{2}{5},\frac{1}{5})$ and $(\frac{1}{2},\frac{1}{2})$, will be produced by interpolation, which further leads to the restriction that $(\frac{1}{p}, \frac{1}{q})$ satisfy $\frac{1}{q}>\frac{3}{p}-1 $. Since there are two cases for obtaining the $L^{p}(\mathbb{R}^{2})\rightarrow L^{p}(\mathbb{R}^{3})$ estimates of $A_{j,k}$, see \eqref{eq:2.27} and \eqref{eq:2.29}, we will establish \eqref{eq:2.103} by considering the following two cases:
\begin{enumerate}
  \item[$\bullet$] Consider $(\frac{1}{p}, \frac{1}{q})\in \{(\frac{1}{p}, \frac{1}{q}) :\ \frac{1}{q} \leq\frac{1}{7p}+\frac{1}{7} \} \cap\Delta$. Let $p:=p_0$ and $q:=q_0$ in \eqref{eq:2.27}, by taking $\delta:=\frac{1}{2p_0}$, we know that \begin{align}\label{eq:2.38}
\left\|A_{j,k}f\right\|_{L^{q_0}(\mathbb{R}^{3})} \lesssim 2^{- \frac{3}{2p_0} k} \|f\|_{L^{p_0}(\mathbb{R}^{2})}
\end{align}
for all $(\frac{1}{p_0},\frac{1}{q_0})$ satisfying $\frac{1}{p_0}=\frac{1}{q_0}\in (0,\frac{1}{6}]$. We remark that \eqref{eq:2.38} implies that \eqref{eq:2.103} is established with $\mu:=\frac{3}{2p}$ for all $\frac{1}{p}=\frac{1}{q}\in (0,\frac{1}{6}]$. Let $p:=p_1$ and $q:=q_1$ in \eqref{eq:2.30}, by letting $\delta:=\frac{1}{q_1}-\frac{1}{2p_1}$, we deduce that
\begin{align}\label{eq:2.39}
\left\|A_{j,k}f\right\|_{L^{q_1}(\mathbb{R}^{3})} \lesssim 2^{- \left(\frac{2}{q_1}-\frac{1}{2p_1}\right)k} \|f\|_{L^{p_1}(\mathbb{R}^{2})}
\end{align}
for all $(\frac{1}{p_1}, \frac{1}{q_1})\in \Delta \cap\{(\frac{1}{p_1}, \frac{1}{q_1}) :\ 0\leq \frac{1}{q_1} \leq \frac{3}{14}, \frac{3}{q_1}\leq 1-\frac{1}{p_1}, \frac{1}{q_1}\leq \frac{3}{5p_1} \}$. For any $(\frac{1}{p}, \frac{1}{q})\in \{(\frac{1}{p}, \frac{1}{q}) :\ \frac{1}{q} \leq\frac{1}{7p}+\frac{1}{7} \} \cap\Delta$, by interpolation, there exists a constant $t_1\in (0,1)$ such that
\begin{align*}\frac{1}{q}=\frac{1-t_1}{q_0}+\frac{t_1}{q_1}~~~~\textrm{and} ~~~~\frac{1}{p}=\frac{1-t_1}{p_0}+\frac{t_1}{p_1}\end{align*}
and
\begin{align}\label{eq:2.40}
\left\|A_{j,k}f\right\|_{L^{q}(\mathbb{R}^{3})} \lesssim 2^{- \frac{3}{2p_0}(1-t_1) k} 2^{- \left(\frac{2}{q_1}-\frac{1}{2p_1}\right)t_1k} \|f\|_{L^{p}(\mathbb{R}^{2})}.
\end{align}
By simple calculation, we see that \eqref{eq:2.40} leads to \eqref{eq:2.103} is established with $\mu:=\frac{3}{2p_0}(1-t_1)+(\frac{2}{q_1}-\frac{1}{2p_1})t_1\in(\frac{1}{q}, \infty)$.
  \item[$\bullet$] Consider $(\frac{1}{p}, \frac{1}{q})\in \{(\frac{1}{p}, \frac{1}{q}) :\ \frac{1}{q} > \frac{1}{7p}+\frac{1}{7} \} \cap\Delta$. Let $p:=p_0$ and $q:=q_0$ in \eqref{eq:2.29}, by setting $\delta:=\frac{1}{12}$, it follows that
    \begin{align}\label{eq:2.41}
\left\|A_{j,k}f\right\|_{L^{q_0}(\mathbb{R}^{3})} \lesssim 2^{- \left(\frac{1}{8}+\frac{3}{4p_0}\right) k} \|f\|_{L^{p_0}(\mathbb{R}^{2})}
\end{align}
for all $(\frac{1}{p_0},\frac{1}{q_0})$ satisfying $\frac{1}{p_0}=\frac{1}{q_0}\in (\frac{1}{6}, \frac{1}{2})$. For any $(\frac{1}{p}, \frac{1}{q})\in \{(\frac{1}{p}, \frac{1}{q}) :\ \frac{1}{q} >\frac{1}{7p}+\frac{1}{7} \} \cap\Delta$, by interpolation between \eqref{eq:2.41} with \eqref{eq:2.39}, there exists a constant $t_2\in (0,1)$ such that
\begin{align*}\frac{1}{q}=\frac{1-t_2}{q_0}+\frac{t_2}{q_1}~~~~\textrm{and} ~~~~\frac{1}{p}=\frac{1-t_2}{p_0}+\frac{t_2}{p_1}\end{align*}
and
\begin{align*}
\left\|A_{j,k}f\right\|_{L^{q}(\mathbb{R}^{3})} \lesssim 2^{- \left(\frac{1}{8}+\frac{3}{4p_0}\right)(1-t_2) k}  2^{- \left(\frac{2}{q_1}-\frac{1}{2p_1}\right)t_2k} \|f\|_{L^{p}(\mathbb{R}^{2})},
\end{align*}
which implies that \eqref{eq:2.103} is established with $\mu:=(\frac{1}{8}+\frac{3}{4p_0})(1-t_2)+(\frac{2}{q_1}-\frac{1}{2p_1})t_2\in(\frac{1}{q}, \infty)$.

\end{enumerate}

Putting things together we have \eqref{eq:2.103}. Therefore, we finish the proof of Theorem \ref{thm1}.

\section{Proof of Theorem \ref{thm2}}

In this section, we consider the necessity of the regions of $(\frac{1}{p}, \frac{1}{q})$ in Theorem \ref{thm1}, i.e., Theorem \ref{thm2}. We split four parts to finish our proofs.
\begin{enumerate}
  \item[$\bullet$] Proof of (i) of Theorem \ref{thm2}. Let $S$ be the $2^{-m}$ neighborhood of the curve $(-t,-\gamma(t))$ with $t\in (0,1]$ and $m\in\mathbb{N}$ large enough. It is easy to see that $(x_1-t, x_2-\gamma(t))\in S$ if $(x_1,x_2)\in B(0, 2^{-m})$ and $t\in(0, 1]$. Furthermore, let us set $f:=\chi_{S}$, we have
      \begin{align*}
      \int_{0}^{1}\chi_{S}(x_1-t,x_2-\gamma(t))\,\textrm{d}t= 1
      \end{align*}
      for any $(x_1,x_2)\in B(0, 2^{-m})$. Then it follows that
      \begin{align*}
      \left\|\int_{0}^{1}\chi_{S}(x_1-t,x_2-\gamma(t))\,\textrm{d}t\right\|_{L^{q}(B(0, 2^{-m}))}\gtrsim 2^{-\frac{2m}{q}} ~~\textrm{and}~~ \|f\|_{L^{p}(\mathbb{R}^{2})}\approx 2^{-\frac{m}{p}}.
      \end{align*}
      Apply the estimate $\|M_{\gamma}f\|_{L^{q}(\mathbb{R}^{2})} \lesssim \|f\|_{L^{p}(\mathbb{R}^{2})}$ to deduce that
      \begin{align*}
      2^{-\frac{2m}{q}} \lesssim 2^{-\frac{m}{p}}.
      \end{align*}
      As a consequence, we have that $(\frac{1}{p},\frac{1}{q})$ satisfy $\frac{1}{2p}\leq\frac{1}{q}$ since $m\in\mathbb{N}$ may tend to infinity. This concludes the proof of (i) of Theorem \ref{thm2}.

  \item[$\bullet$] Proof of (ii) of Theorem \ref{thm2}. Notice that $M_{\gamma}$ commutes with translations, then it is not bounded from $L^p(\mathbb{R}^2)$ to $L^q(\mathbb{R}^2)$ for all $p>q$. This, in combination with the estimate $\|M_{\gamma}f\|_{L^{q}(\mathbb{R}^{2})} \lesssim \|f\|_{L^{p}(\mathbb{R}^{2})}$, shows that $(\frac{1}{p},\frac{1}{q})$ satisfy $\frac{1}{q}\leq\frac{1}{p}$. Hence, we obtain (ii) of Theorem \ref{thm2}.

  \item[$\bullet$] Proof of (iii) of Theorem \ref{thm2}. One may assume without loss of generality that $\gamma(1)=1$. Let $e_1:=(-\frac{1}{\sqrt{1+\gamma'(1)^2}}, -\frac{\gamma'(1)}{\sqrt{1+\gamma'(1)^2}})$ and $e_2:=(-\frac{\gamma'(1)}{\sqrt{1+\gamma'(1)^2}}, \frac{1}{\sqrt{1+\gamma'(1)^2}})$ be two orthogonal unit vectors, and define
      $$S:=\left\{(x_1,x_2)\in \mathbb{R}^{2}:\ |(x_1,x_2)\cdot e_1|\leq 5 \cdot2^{-m} ~~\textrm{and}~~ |(x_1,x_2)\cdot e_2|\leq (1+2\gamma''(1)) 2^{-2m} \right\}.$$
      We choose a subset $\{u_i\}\subset [1,2]$ such that $|u_{i+1}-u_i|=2\cdot 2^{-2m}$ for all $i$. For each $u_i$, define
      $$D_{u_i}:=\left\{(x_1,x_2)\in \mathbb{R}^{2}:\ |\left((x_1,x_2)-(u_i,u_i)\right)\cdot e_1|\leq 2^{-m} ~\textrm{and}~ |\left((x_1,x_2)-(u_i,u_i)\right)\cdot e_2|\leq 2^{-2m} \right\}.$$
      For any $(x_1,x_2)\in D_{u_i}$ and $t\in[1-\frac{2^{-m}}{\sqrt{1+\gamma'(1)^2}}, 1]$, applying Taylor's theorem gives
      $$\gamma(t)=1+\gamma'(1)(t-1)+\frac{\gamma''(1)}{2}(t-1)^2+o\left((t-1)^2\right),$$
      where $o((t-1)^2)$ means $\frac{o((t-1)^2)}{(t-1)^2}\rightarrow 0$ as $t\rightarrow 1^-$. Furthermore, by simple calculation, we have
      \begin{align*}
      &\left|\left((u_it,u_i\gamma(t))-(u_i,u_i)\right)\cdot e_1\right|\\
      =&\left|-u_i\sqrt{1+\gamma'(1)^2}(t-1)-\frac{u_i\gamma'(1)\gamma''(1)}{2\sqrt{1+\gamma'(1)^2}} (t-1)^2- \frac{u_i\gamma'(1)}{\sqrt{1+\gamma'(1)^2}}o\left((t-1)^2\right) \right|\leq 4 \cdot2^{-m}
      \end{align*}
      and
      \begin{align*}
      &\left|\left((u_it,u_i\gamma(t))-(u_i,u_i)\right)\cdot e_2\right|\\
      =&\left|\frac{u_i\gamma''(1)}{2\sqrt{1+\gamma'(1)^2}} (t-1)^2+ \frac{u_i}{\sqrt{1+\gamma'(1)^2}}o\left((t-1)^2\right) \right|\leq 2\gamma''(1)2^{-2m}
      \end{align*}
      for $m\in\mathbb{N}$ large enough. By the triangle inequality, we may conclude that
      \begin{align*}
      &\left|\left((x_1,x_2)-(u_it,u_i\gamma(t))\right)\cdot e_1\right|\\
      \leq&\left|\left((x_1,x_2)-(u_i,u_i)\right)\cdot e_1\right|+\left|\left((u_it,u_i\gamma(t))-(u_i,u_i)\right)\cdot e_1\right|\leq 5 \cdot2^{-m}
      \end{align*}
      and
       \begin{align*}
      &\left|\left((x_1,x_2)-(u_it,u_i\gamma(t))\right)\cdot e_2\right|\\
      \leq&\left|\left((x_1,x_2)-(u_i,u_i)\right)\cdot e_2\right|+\left|\left((u_it,u_i\gamma(t))-(u_i,u_i)\right)\cdot e_2\right|\leq (1+2\gamma''(1)) 2^{-2m}
      \end{align*}
      for $m\in\mathbb{N}$ large enough. The above implies $(x_1-u_it, x_2-u_i\gamma(t))\in S$ if $(x_1,x_2)\in D_{u_i}$ and $t\in[1-\frac{2^{-m}}{\sqrt{1+\gamma'(1)^2}}, 1]$. Consequently, let $f:=\chi_{S}$, it follows that
       \begin{align*}
     M_{\gamma}f(x_1,x_2)\geq \int_{1-\frac{2^{-m}}{\sqrt{1+\gamma'(1)^2}}}^1\chi_{S}(x_1-u_it, x_2-u_i\gamma(t))\,\textrm{d}t= \frac{2^{-m}}{\sqrt{1+\gamma'(1)^2}}
      \end{align*}
      for any $(x_1,x_2)\in D_{u_i}$ and $m\in\mathbb{N}$ large enough. Furthermore, notice that for any $i\neq i'$, $D_{u_i}\cap D_{u_{i'}}=\emptyset$, and the number of the elements in $\{u_i\}\subset [1,2]$ is equivalent to $2^{2m}$, one may obtain
      $$\|M_{\gamma}f\|_{L^{q}(\mathbb{R}^{2})} \gtrsim \left(\sum_i \int_{D_{u_i}} \left|M_{\gamma}\chi_{S}(x_1,x_2)\right|^q \,\textrm{d}x_1\,\textrm{d}x_2\right)^{\frac{1}{q}}\gtrsim 2^{-m} 2^{-\frac{m}{q}}~~\textrm{and}~~ \|f\|_{L^{p}(\mathbb{R}^{2})}\approx 2^{-\frac{3m}{p}}.$$
      Then the estimate $\|M_{\gamma}f\|_{L^{q}(\mathbb{R}^{2})} \lesssim \|f\|_{L^{p}(\mathbb{R}^{2})}$ implies
      \begin{align*}
      2^{-m} 2^{-\frac{m}{q}} \lesssim 2^{-\frac{3m}{p}},
      \end{align*}
     which further implies that $(\frac{1}{p},\frac{1}{q})$ satisfy $\frac{1}{q}\geq\frac{3}{p}-1$ since $m\in\mathbb{N}$ can be sufficiently large. We may therefore establish (iii) of Theorem \ref{thm2}.

  \item[$\bullet$] Proof of (iv) of Theorem \ref{thm2}. Recall that $\omega=\limsup_{t\rightarrow 0^{+}}\frac{\ln|\gamma(t)|}{\ln t}\in (0,\infty)$, which implies $\omega=\inf_{\tau\in (0,1)}\sup_{t\in (0,\tau)}\frac{\ln|\gamma(t)|}{\ln t}$. Then, for any $\epsilon>0$, there exists a positive constant $\tau\in (0,1)$ such that $\sup_{t\in (0,\tau)}\frac{\ln|\gamma(t)|}{\ln t}\in [\omega,\omega+\epsilon)$. Moreover, there exists a positive constant $t\in (0,\tau)$ such that $\frac{\ln|\gamma(t)|}{\ln t}\in (\omega-\epsilon,\omega+\epsilon)$, which further implies that
      $$\sup_{s\in (0,t)}\frac{\ln|\gamma(s)|}{\ln s}\leq\sup_{s\in (0,\tau)}\frac{\ln|\gamma(s)|}{\ln s}< \omega+\epsilon ~~\textrm{and}~~ \sup_{s\in (0,t)}\frac{\ln|\gamma(s)|}{\ln s}\geq \omega. $$
      Hence, there exists a positive constant $t'\in (0,t)$ such that $\frac{\ln|\gamma(t')|}{\ln t'}\in (\omega-\epsilon,\omega+\epsilon)$. Repeating the same procedure, there exists a sequence $\{t_i\}_{i=1}^{\infty}\subset (0,1)$ such that $\frac{\ln|\gamma(t_i)|}{\ln t_i}\in (\omega-\epsilon,\omega+\epsilon)$, where $\{t_i\}_{i=1}^{\infty}$ is strictly decreasing as $i$ increases and satisfies $\lim_{i\rightarrow\infty}t_i=0$ . We can then write
      \begin{align}\label{eq:3.1a}
      \left|\gamma(t_i)\right|={t_i}^{\log_{t_i} {|\gamma(t_i)|}}={t_i}^{\frac{\ln |\gamma(t_i)|}{\ln t_i}}\in \left({t_i}^{\omega+\epsilon}, {t_i}^{\omega-\epsilon}\right).
      \end{align}
      \quad On the other hand, we can prove that
      \begin{align}\label{eq:3.1b}
      t_i\left(t_i\left|\gamma(t_i)\right|\right)^{\frac{1}{q}-\frac{1}{p}}\lesssim 1,
      \end{align}
      where the implicit constant is independent of $i$. Indeed, let $S:=[-t_i, t_i]\times[-2|\gamma(t_i)|, 2|\gamma(t_i)|]$ and $D:=[0, t_i]\times[0, |\gamma(t_i)|]$, it is not difficult to check that $(x_1-t, x_2-\gamma(t))\in S$ for all $(x_1,x_2)\in D$ and $t\in[0, t_i]$. Hence, let $f:=\chi_{S}$, one obtains
      \begin{align*}
      \int_{0}^{t_i}\chi_{S}(x_1-t,x_2-\gamma(t))\,\textrm{d}t= t_i
      \end{align*}
      for any $(x_1,x_2)\in D$, it then follows that
      \begin{align*}
      \left\|\int_{0}^{t_i}\chi_{S}(x_1-t,x_2-\gamma(t))\,\textrm{d}t\right\|_{L^{q}(D)}\gtrsim t_i\left(t_i\left|\gamma(t_i)\right|\right)^{\frac{1}{q}} ~~\textrm{and}~~ \|f\|_{L^{p}(\mathbb{R}^{2})}\approx \left(t_i\left|\gamma(t_i)\right|\right)^{\frac{1}{p}}.
      \end{align*}
     This, combined with the estimate $\|M_{\gamma}f\|_{L^{q}(\mathbb{R}^{2})} \lesssim \|f\|_{L^{p}(\mathbb{R}^{2})}$, implies \eqref{eq:3.1b} as desired.

     \quad By (ii) of Theorem \ref{thm2}, we have that $(\frac{1}{p},\frac{1}{q})$ satisfy $\frac{1}{q}\leq\frac{1}{p}$. Combining \eqref{eq:3.1a} and \eqref{eq:3.1b}, we obtain
      \begin{align*}
      1\gtrsim t_i\left(t_i\left|\gamma(t_i)\right|\right)^{\frac{1}{q}-\frac{1}{p}} \gtrsim t_i\left(t_i\left({t_i}^{\omega-\epsilon}\right)\right)^{\frac{1}{q}-\frac{1}{p}}={t_i}^{1+(1 +\omega-\epsilon)\left(\frac{1}{q}-\frac{1}{p}\right)}.
      \end{align*}
      Notice that the sequence $\{t_i\}_{i=1}^{\infty}\subset (0,1)$ is strictly decreasing as $i$ increases and satisfies $\lim_{i\rightarrow\infty}t_i=0$, we must have
      \begin{align}\label{eq:3.1c}
      1+(1 +\omega-\epsilon)\left(\frac{1}{q}-\frac{1}{p}\right)\geq 0
      \end{align}
      for all $\epsilon>0$. Let $\epsilon\rightarrow0$ in \eqref{eq:3.1c}, it now follows that $(\frac{1}{p},\frac{1}{q})$ satisfy $1+(1 +\omega)\left(\frac{1}{q}-\frac{1}{p}\right)\geq 0$, and so we have (iv) of Theorem \ref{thm2}.
\end{enumerate}

Therefore, putting things together we finish the proof of Theorem \ref{thm2}.

\bigskip

\noindent{\bf Acknowledgments.} The authors would like to thank Lixin Yan for helpful suggestions and discussions. We also particularly thank Shaoming Guo for his constant support and providing us some important references which is very useful to our work.

\bigskip

\noindent  Naijia Liu

\smallskip

\noindent  School of Mathematics, Sun Yat-sen University, Guangzhou, 510275,  People's Republic of China

\smallskip

\noindent {\it E-mail}: \texttt{liunj@mail2.sysu.edu.cn}

\bigskip

\noindent Haixia Yu (Corresponding author)

\smallskip

\noindent  Department of Mathematics, Shantou University, Shantou, 515063, People's Republic of China

\smallskip

\noindent{\it E-mail}: \texttt{hxyu@stu.edu.cn}

\bigskip

\end{document}